\documentclass{imanumArXiv}

\usepackage{graphicx,wrapfig,lipsum}

\usepackage{chngcntr}
\usepackage{epsfig}
\usepackage{mathtools}
\theoremstyle{plain}
\newtheorem*{theorem*}{Theorem}
\usepackage[utf8]{inputenc}
\usepackage[english]{babel}
\usepackage[utf8]{inputenc}
\usepackage{epsfig}
\usepackage{caption}
\usepackage{graphicx}
\usepackage{subcaption}
\usepackage{xcolor}
\usepackage{mhchem}
\usepackage{chemformula}
\usepackage{siunitx}
\usepackage{authblk}
\usepackage[figurename=FIG.]{caption}
\usepackage[tablename=TABLE]{caption}
\usepackage[labelsep=space]{caption}
\usepackage{natbib}
\usepackage{url}
\numberwithin{equation}{section}

\def\restrict#1{\raise-.5ex\hbox{\ensuremath|}_{#1}}

\begin{document}
 \title{\textbf{Augmented saddle point formulation of the steady-state Stefan--Maxwell diffusion problem}}
 \date{\today}
 
 \author{%
{\sc
Alexander Van-Brunt \thanks{Corresponding author. Email: alexander.van-brunt@maths.ox.ac.uk},
} \\
Mathematical Institute, University of Oxford, Oxford, OX2 6GG, UK and \\ The Faraday Institution, Harwell Campus,  Didcot, OX11 ORA, UK \\
{\sc Patrick E.~Farrell \thanks{Email: patrick.farrell@maths.ox.ac.uk}}\\
Mathematical Institute, University of Oxford, Oxford, OX2 6GG, UK \\
{\sc and}\\
{\sc Charles W.~Monroe \thanks{Email: charles.monroe@eng.ox.ac.uk}}\\
Department of Engineering Science, University of Oxford, Oxford OX1 3PJ, UK and \\ The Faraday Institution, Harwell Campus,  Didcot, OX11 ORA, UK
}

% \author[1,2, $\dagger$]{Alexander Van-Brunt}
% \author[1, $\ddagger$]{Patrick E.~Farrell}
% \author[2,3, $\star$]{Charles Monroe}
% \affil[1]{\textit{Mathematical Institute, University of Oxford, Woodstock Road, Oxford, OX2 6GG, UK}}
% \affil[2]{\textit{The Faraday Institution, Harwell Campus,  Didcot, OX11 ORA, UK}}
%\affil[3]{\textit{Department of Engineering Science, University of Oxford, Oxford OX1 3PJ, UK}}
% \affil[$\dagger$]{\normalsize{alexander.van-brunt@maths.ox.ac.uk} (corresponding author)} 
% \affil[$\ddagger$]{\normalsize{patrick.farrell@maths.ox.ac.uk}} 
% \affil[$\star$]{charles.monroe@eng.ox.ac.uk} 
 \maketitle

\begin{abstract}
{We investigate structure-preserving finite element discretizations of the steady-state Stefan--Maxwell diffusion problem which governs diffusion within a phase consisting of multiple species. An approach inspired by augmented Lagrangian methods allows us to construct a symmetric positive definite augmented Onsager transport matrix, which in turn leads to an effective numerical algorithm. We prove inf-sup conditions for the continuous and discrete linearized systems and obtain error estimates for a phase consisting of an arbitrary number of species. The discretization preserves the thermodynamically fundamental Gibbs--Duhem equation to machine precision independent of mesh size. The results are illustrated with numerical examples, including an application to modelling the diffusion of oxygen, carbon dioxide, water vapour and nitrogen in the lungs.}
{Stefan--Maxwell equations, multicomponent diffusion, augmented saddle point formulation}
 \end{abstract}

\section{Introduction} 

 Molecular diffusion is a fundamental mode of mass transport. Within a stationary solution containing a dilute solute species of concentration $c$, the classical model for diffusion was formulated by \cite{Fick}, which postulates that the solute's molar flux $N$ obeys
\begin{equation} \label{Fickianansatz}
N = -D \nabla c,
\end{equation}
 in which $D>0$ is the solute's Fickian diffusivity in the solution. \cite{Maxwell1867} applied kinetic theory to derive Fick's law for binary ideal-gas diffusion, showing for isothermal gases that $D$ further relates to a composition-independent constant material property. \cite{Stefan1871} extended Maxwell's analysis to multicomponent gases, expressing the gradient of each species concentration in terms of a matrix of binary diffusivities. The resulting Stefan--Maxwell equations (also commonly called Maxwell--Stefan equations in the engineering and mathematical literature) have been verified experimentally for gas diffusion in studies such as \cite{Duncan1969} and \cite{1975Carty}.

Using his theory of irreversible thermodynamics, \cite{Onsager1931first, Onsager1931second, Onsager1945} %showed how the non-negative-definite, symmetric form of the transport matrix for multicomponent mass diffusion follows from a variational analysis of the macroscopic functional that describes local entropy production \cite{Onsager1945}. This approach 
provided a broader theoretical framework for mass transport that could also be applied to multicomponent diffusion in nonideal phases, including liquids and/or solids. \cite{hirschfelder1954molecular} substantiated this more abstract analysis, manipulating thermodynamic laws and hydrodynamic equations to construct the diffusion driving forces for general nonisobaric, nonisothermal, multicomponent diffusion systems. Combined with Lightfoot, Cussler and Rettig's observation that the Stefan--Maxwell diffusivities map invertibly into Onsager's transport matrix, and can therefore be used for condensed phases as well as gases \citep{Lightfoot1962}, this extended the Stefan--Maxwell theory to all molecular diffusion processes. \cite{Newman1965} brought the generalization further, accounting for materials containing charged solutes, thereby completing the development of the contemporary Stefan--Maxwell equations. 
Modern expositions of the theory can be found in \cite{KRISHNA1997861, Krishna1979} and \cite{ DATTA20105976}.

Given a bounded Lipschitz domain $\Omega \subset \mathbb{R}^{d}, \: d \in \{2,3 \}$, the Stefan--Maxwell equations describing the
diffusion of the $n$ species that constitute a common phase at a given absolute temperature $T > 0$ are given by
\begin{equation} \label{FundamentalMSeqn}
d_{i} = \sum_{\substack{j=1 \\ i\neq j}}^{n} \frac{RT c_{i} c_{j}}{ \mathcal{D}_{ij} c_{\text{T}}} (v_{i} - v_{j})
\end{equation}
for $i = 1,2,\dots, n$, in which $R > 0$ is the ideal gas constant. 
The terms $c_{i}:  \Omega \rightarrow \mathbb{R}^{+}$ and $v_{i}: \Omega \rightarrow \mathbb{R}^{d}$ denote the concentration and velocity of species $i= 1,2,\dots,n$ respectively, %These terms are 
related to the molar flux of the $i^{\mathrm{th}}$ species by $N_{i} = c_{i} v_{i}$. For species $i \neq j$, $\mathcal{D}_{ij} \in \mathbb{R}$ represents the Stefan--Maxwell diffusivity of species $i$ through species $j$; these material parameters are symmetric in the species indices, $\mathcal{D}_{ij} = \mathcal{D}_{ji}$, and coefficients $\mathcal{D}_{ii}$ are not defined. %The term $c_{\text{T}}$ in equation \eqref{FundamentalMSeqn} denotes the total concentration, defined as 
The term $c_\text{T}$ in equation \eqref{FundamentalMSeqn} denotes the total concentration, defined as
\begin{equation} \label{Totalconcentrationdef}
c_{\text{T}}: = \sum_{i=1}^{n} c_{i},
\end{equation}
and the terms $d_{i}: \Omega \rightarrow \mathbb{R}^{d}$ represent the diffusion driving forces, which generally depend on the species concentrations, temperature and pressure.  In the case of isothermal, isobaric ideal-gas diffusion considered here, $d_{i} = - RT \nabla c_{i}$. Furthermore, an ideal gas satisfies the equation of state 
\begin{equation} \label{Idealgaslaw}
p = c_{\text{T}} RT,
\end{equation}
in which $p$ is the pressure. Hence in the  isothermal, isobaric setting, $c_{\text{T}}$ is a constant.  To pose the Stefan--Maxwell convection-diffusion problem,  flux constitutive laws \eqref{FundamentalMSeqn} and the equation of state \eqref{Idealgaslaw} are coupled to the continuity equations
\begin{equation} \label{Fundamentalcontinuityeqn}
\frac{\partial c_{i}}{\partial t} = -\nabla \cdot \big (c_{i} v_{i} \big)  + r_{i},
\end{equation}
where $r_{i}: \Omega \rightarrow \mathbb{R}$ is a specified volumetric reaction rate, which quantifies the generation or depletion of  species $i$ by homogeneous chemical reactions. 
Under the ideal-gas assumption considered in this paper, we are interested in solving \eqref{FundamentalMSeqn} (with $d_{i} = -RT \nabla c_{i}$) and \eqref{Fundamentalcontinuityeqn} for the species concentrations $c_{i}$ and their respective velocities $v_{i}$.

The Stefan--Maxwell equations have found a diverse range of applications in areas such as biology, electrochemistry, and plasma physics. For specific examples, we refer the reader to the studies by \cite{Boudin2010, Bruce1988, ABDULLAH20075821, newman2012electrochemical, LIU2014447} and \cite{Maxwell--Stefan-plasma}. Because they account for solute/solute interactions as well as solute/solvent interactions, Stefan--Maxwell models can exhibit fundamentally different behaviour from Fickian models. For example, in \cite{Duncan1969} `uphill diffusion' is observed, wherein 
the directions of a species' molar flux and its concentration gradient coincide, in contradiction to \eqref{Fickianansatz}. In electrochemistry, the Stefan--Maxwell formalism justifies surprising observations like negative transference, where the flow of an electric current with one sign carries ions of opposing sign along with it \citep{Monroe2013}.

%In certain regimes it is possible to recover Fick's law. %For example, when $d_{i} = - RT\nabla c_{i}$ and $v_i = 0$ for $i\geq 2$, equation \eqref{FundamentalMSeqn} simplifies to
 %For example, in the case of binary ideal-gas diffusion, direct inversion of the single independent Maxwell--Stefan equation following Helfand \cite{Helfand1960} shows that
%\begin{equation}
%N_1 = - \left( 1 + \frac{c_1}{c_2} \right) \mathcal{D}_{12} \nabla c_1 + \frac{c_1}{c_2} N_2.
%\end{equation}
%Sometimes the constancy of $c_\textrm{T}$ is taken to justify an assumption of equimolar counterdiffusion, so that $N_1 = -N_2$ everywhere, in which case
%\begin{equation}
%-\nabla c_{1} =  \Big (\sum_{j=2}^{n} \frac{c_{j}}{\mathcal{D}_{1j}} \Big ) N_{1}.
%N_1 = - \mathcal{D}_{12} \nabla c_1 .
%\end{equation}
Under restrictive assumptions, a multicomponent extension of Fick's law known as \emph{dilute solution theory} can be recovered. For a dilute set of species $c_{k}, \: k = 2,\dots,n$ in the presence of a solvent in far greater proportions, $c_{1}$, one can formally neglect the terms $c_{i}c_{j}$ in  \eqref{FundamentalMSeqn} whenever both $i,j \geq2$ and take $c_\textrm{T} \approx c_1$, allowing rearrangement to express the molar fluxes as
\begin{equation} \label{diluteNs}
N_{i} = \frac{\mathcal{D}_{i1}}{RT} d_i + c_{i} v_1.
\end{equation}
 For dilute solutes the driving forces often take the form $d_i = - \chi_i RT \nabla c_i$, where $\chi_{i}$ is known as a Darken factor \citep{Darken48}, in which case one can identify  $D_{i} = \mathcal{D}_{i1} \chi_i$ as the Fickian diffusivity of species $i$ in the solution. Writing equations \eqref{diluteNs} for all solutes and replacing $v_1$ with the barycentric velocity $v$ produces the dilute solution theory. We direct the reader to \cite{newman2012electrochemical} for further details. When the solute driving forces within dilute solution theory are written in terms of both concentration gradients and the electric field, \eqref{diluteNs} is referred to as a Nernst--Planck relationship, based on the work by \cite{Nernst1888} and \cite{Planck1890}. Nernst--Planck equations have been extensively studied in the mathematical literature, sometimes coupled with Poisson's equation to account for the distribution of the electric potential, and a Navier--Stokes equation or a Darcy flow to compute the velocity $v$. We refer the reader to \cite{2009Schmuck, EUDNPPequation, Jinchao2015} and \cite{Jinchao2018} and the references therein for the existing mathematical literature on the Nernst--Planck equation. In many cases dilute solution theory is not appropriate, and the full Stefan--Maxwell equations must be considered. A comparison of Fickian to Stefan--Maxwell diffusion profiles for gases can be found in \cite{KRISHNA1997861} and \cite{Boudin2010}. Examples of the limitations of dilute-solution models are discussed in the context of lung modelling, earth science and electrolyte transport in the studies by \cite{CHANG1975109, Arthur1990} and \cite{MonroeNernstPlanck2016} respectively.

\subsection{Physical structure and consequences}

The mathematical structure of irreversible thermodynamics will prove useful below for devising discretizations and error estimates for the Stefan--Maxwell diffusion problem. We therefore summarize some key points of the theory. We begin with the transport equations postulated by \cite{Onsager1945} for isotropic  materials,
\begin{equation} \label{TransportmatrixMS}
d_{i} = \sum_{j=1}^{n}\textbf{M}_{ij} v_{j},
\end{equation}
in which the statistical reciprocal relations developed in \cite{Onsager1931first, Onsager1931second} require the transport matrix $\textbf{M}: \Omega \rightarrow \mathbb{R}^{n \times n}$ to be real symmetric. 

The Onsager transport equations \eqref{TransportmatrixMS} were developed independently of the Stefan--Maxwell theory \eqref{FundamentalMSeqn}. It was subsequently realized by \cite{Lightfoot1962} that the Stefan--Maxwell equations could be understood in terms of Onsager's transport matrix by identifying 
\begin{equation} \label{transport matrix}
\textbf{M}_{ij} =\textbf{M}_{ij}(c_{i},c_{j}, c_\text{T}) = \begin{cases} -\frac{RT c_{i} c_{j}}{\mathcal{D}_{ij} c_{\text{T}}}  \:\:\: \text{if} \:\: i \neq j \\
\sum_{k \neq i}^{n}\frac{RT c_{i} c_{k}}{\mathcal{D}_{ik} c_{\text{T}}} \:\:\: \text{if} \:\: i = j \end{cases}
\end{equation}
 as the entries of $\textbf{M}$.

 Time evolution of nonequilibrium states leads to local entropy production, denoted by $\dot{S}$. For an isothermal, isobaric system with a given collection of species velocities $(v_{1},v_{2},\dots,v_{n})$ experiencing $(d_{1},d_{2},\dots,d_{n})$ nonequilibrium diffusion driving forces, the balances of material, momentum, and heat are manipulated in  \cite{hirschfelder1954molecular} and \cite{Monroe2009} to write the entropy production of isothermal diffusion as 
\begin{equation} \label{Entropyproduction}
\dot{S} = \frac{1}{T} \sum_{i=1}^{n} d_{i} \cdot v_{i},
\end{equation}
 in which the $d_i$ are identified as functions of the gradients of temperature, pressure, and $(c_1, c_2,...,c_n)$ by grouping terms in the Gibbs--Duhem equation from equilibrium thermodynamics \citep{deGrootMazur, hirschfelder1954molecular, Goyal01012017}. In general there may be other terms in \eqref{Entropyproduction} such as viscous dissipation or reaction entropy, but these will be neglected here. 

 The second law of thermodynamics further demands that the energy dissipation $T \dot{S}$ is non-negative, $T\dot{S} \geq 0$, with equality only in an equilibrium state, which is defined by the condition that $d_i = 0$ for all $i$. Thermodynamic stability therefore requires further that $\textbf{M}$ be positive semidefinite.

Some additional structure of the transport matrix is specific to multicomponent mass diffusion. Importantly, the theory must guarantee that diffusional motion, driven by thermodynamic property gradients, remains distinct from species convection, a non-dissipative process driven by bulk flow. This distinction is made by requiring that \eqref{TransportmatrixMS} be invariant to a shift of every species velocity by a vector field $\bar{v}: \Omega \rightarrow \mathbb{R}^d$, i.e.~the equation remains unchanged when each $v_{i}$ in \eqref{TransportmatrixMS} is replaced by $(v_{i} - \bar{v})$. The essential physical distinction between diffusion and convection consequently requires that 
\begin{equation} \label{Nullspace}
\sum_{j=1}^{n} \textbf{M}_{ij} = 0,
\end{equation}
as noted by \cite{Onsager1945} and \cite{Helfand1960}. Hence $\textbf{M}$ has a null eigenvalue corresponding to the eigenvector $(1,1,\dots,1)^{\top} \in \mathbb{R}^{n}$. Invariance with respect to the convective velocity is naturally embedded in the Stefan--Maxwell form \eqref{FundamentalMSeqn}, because $v_{i}-v_{j}=(v_{i} - \bar{v})-(v_{j} -\bar{v})$. 

The symmetry of $\textbf{M}$ suggested by \cite{Onsager1945} requires that $\mathcal{D}_{ij} = \mathcal{D}_{ji}$, a fact that has also been demonstrated directly for Stefan--Maxwell diffusion by fluctuation theory in \cite{Monroe2015}. The symmetry of the transport matrix combined with the nullspace \eqref{Nullspace} allows recovery of the full \text{Gibbs--Duhem equation}, namely
\begin{equation} \label{FundamentalGibbsD}
\sum_{i=1}^{n} d_{i} = 0.
\end{equation}
In the context of transport theory, equation \eqref{FundamentalGibbsD} can be seen as a statement of Newton's third law of motion, that action equals reaction. In thermodynamics this is necessary to be consistent with the first law of thermodynamics and the extensivity of the Gibbs free energy.

Reasoning physically that all diffusion processes are necessarily dissipative, \cite{Onsager1945} makes the stronger assumption that $\textbf{M}$ has exactly one null eigenvalue.  %(Situations may arise where local chemical reaction equilibria exist within a system, introducing dependences among the collection of driving forces, but such constraints do not affect the structure on the spectrum of the transport matrix.) 
Taken together, the physical arguments require that $\textbf{M}$ is symmetric positive semidefinite, and that its eigenvalues, $\{ \lambda_{i=1}^{\textbf{M}} \}_{i=1}^{n}$, may be ordered as
\begin{equation} \label{SpectrumofM}
0 = \lambda^{\textbf{M}}_{1} < \lambda^{\textbf{M}}_{2} \leq \dots \leq \lambda_{n}^{\textbf{M}},
\end{equation}
a spectral structure that will be used throughout this paper.

Combining \eqref{TransportmatrixMS} and \eqref{Entropyproduction} implies that
\begin{equation} \label{PSD}
T \dot{S} = \sum_{i,j=1}^{n}  v_{i} \cdot \textbf{M}_{ij} v_{j} =\frac{1}{2}\sum_{i=1}^{n} \sum_{j \neq i}^{n} \frac{c_{i} c_{j} R T}{\mathcal{D}_{ij} c_{\text{T}}} (v_{j} - v_{i})^{2}
  \geq 0.
\end{equation}
At positive concentrations, energy dissipation $T \dot{S} > 0$ occurs whenever there is relative species motion, implying that the equality in \eqref{PSD} occurs if and only if $v_1 = v_2 = \cdots = v_n$. 

One must take care to note that $\textbf{M}$ may afford additional nullspaces beyond \eqref{Nullspace} if any concentration vanishes. Consequently, in order to phrase the Stefan--Maxwell equations in terms of Onsager's transport laws \eqref{TransportmatrixMS} with a transport matrix $\textbf{M}$ that possesses the spectral structure \eqref{SpectrumofM}, it will be necessary to assume that $c_{i} >0$ almost everywhere for each $i = 1,2,\dots n$. We make this assumption henceforth.

Because the present discussion is limited to ideal-gas mixtures, %For simplicity 
it can be assumed that the Stefan--Maxwell diffusion coefficients $\mathcal{D}_{ij}$ are given constants, which places even stronger restrictions on their values.   Whenever the concentrations satisfy $c_{i} \geq \kappa >0$ for each $i=1,2,\dots,n$ and any positive constant $\kappa$, then $\lambda_{\kappa} \leq \lambda_{2}^{\textbf{M}}$ for a positive constant $\lambda_{\kappa}$ which depends only on $\kappa$, a fact that will be used throughout the paper. From the calculation \eqref{PSD}, it follows that a necessary and sufficient condition for \eqref{SpectrumofM} to be true for all positive concentrations is that each $\mathcal{D}_{ij}$ is strictly positive \citep{Krishna1979}. It must be stressed, however, that the Stefan--Maxwell diffusion coefficients in many physical systems depend strongly on the concentrations of the species, in which case negative Stefan--Maxwell diffusion coefficients  are not only possible, but are observed and of practical interest \citep{Gerrit1993,NegativeMScoefficients}. Therefore in order to present a general framework for multispecies diffusion, the results in this paper only use the spectral structure \eqref{SpectrumofM}, not the positivity of the Stefan--Maxwell diffusion coefficients.

In systems with more than one spatial dimension, the existence of the nullspace \eqref{Nullspace} means that the problem \eqref{FundamentalMSeqn}, \eqref{Totalconcentrationdef}, \eqref{Fundamentalcontinuityeqn} will not be well-posed unless a choice of convective velocity is made (see Remark \ref{1Dcase} below for the one-dimensional case). This can be done by specifying that the mass-flux must equal given data $u: \Omega \to \mathbb{R}^d$:
\begin{equation} \label{mass-flux constraint}
 u = \sum_{j=1}^{n} M_{j} c_{j}  v_{j},
\end{equation}
where $M_i > 0$ is the molar mass of species $i$. 
In general, the mass-flux must also be solved for via the Cauchy momentum equation, which in the absence of a 
pressure gradient or an external force field, can be written in conservation form as
\begin{equation} \label{FundamentalCauchy}
\frac{\partial u}{\partial t} = - \nabla \cdot \Big (\rho^{-1} u \otimes u - \sigma \Big ),
\end{equation}
where the density $\rho$ is defined as
\begin{equation} \label{definition of density}
\rho : = \sum_{j=1}^{n} M_{j} c_{j},
\end{equation}
and $\sigma$ denotes the deformation stress tensor appropriate for the medium. We refer to the problem of solving \eqref{FundamentalMSeqn}, \eqref{Fundamentalcontinuityeqn}, \eqref{mass-flux constraint} and \eqref{FundamentalCauchy} as the Stefan--Maxwell convection-diffusion problem. In this work we assume
that $u$ is given and focus on the solution of \eqref{FundamentalMSeqn}, \eqref{Fundamentalcontinuityeqn} and \eqref{mass-flux constraint} under an additional steady-state assumption, which we call the steady-state Stefan--Maxwell diffusion problem.

\subsection{Premise and main results}

The central idea of this manuscript is to incorporate the constraint \eqref{mass-flux constraint} by augmenting \eqref{FundamentalMSeqn}, in
a manner inspired by the augmented Lagrangian approach \citep{Bochev2006,Fortin1983}. 
Given $\gamma >0$, for each $i$ we multiply both sides of \eqref{mass-flux constraint} by $\gamma RT M_{i} c_{i} / \rho$ and add the resulting term to the $i^{\mathrm{th}}$ equation of \eqref{FundamentalMSeqn} to deduce that
\begin{equation} \label{OurMSeqn}
d_{i} + \frac{\gamma RT M_{i} c_i}{\rho} u  =  \sum_{j \neq i}^{n} \frac{RTc_{i} c_{j}}{\mathcal{D}_{ij} c_{\text{T}}} \big ( v_{i} -  v_{j} \big ) + \frac{\gamma RT M_{i} c_i}{\rho}  \sum_{j=1}^{n}M_{j} c_{j}  v_{j} =  \sum_{j=1}^{n} \textbf{M}^{\gamma}_{ij} v_{j}
\end{equation}
for $i = 1,2,\dots, n$, where $\textbf{M}_{ij}^{\gamma}$ is the augmented transport matrix
\begin{equation} \label{augmented transport matrix}
\textbf{M}^{\gamma}_{ij} =\textbf{M}_{ij} + \gamma \mathcal{L}_{ij},
\end{equation}
in which
\begin{equation} \label{definition of L}
\mathcal{L}_{ij}: =  RT M_{i} M_{j} c_{i} c_{j}/ \rho.
\end{equation}
Our particular choice of the entries of $\mathcal{L}$ allows us to compute
\begin{equation} \label{Coerciveness}
\sum_{i,j=1}^{n} v_{i} \cdot \textbf{M}^{\gamma}_{ij} v_{j},= \frac{1}{2}\sum_{i=1}^{n} \sum_{j \neq i}^{n} \frac{c_{i} c_{j} R T}{\mathcal{D}_{ij}c_{\text{T}}} (v_{j} - v_{i})^{2} + \gamma \Big ( \sum_{j=1}^{n} M_{j} c_{j}  v_{j} \Big )^{2}
\end{equation}
to show that the augmented transport matrix is symmetric positive definite. The positive-definiteness achieved by this augmentation will cause the associated bilinear forms in the variational formulation to follow to be coercive, greatly facilitating the analysis. 

The paper is organized as follows. Section $2$ provides an overview of the existing numerical literature on the Stefan--Maxwell equations and contrasts our approach with previous efforts. In section $3$ we derive a suitable weak formulation for the problem and prove well-posedness of a linearized system of \eqref{FundamentalMSeqn}-\eqref{Fundamentalcontinuityeqn} in section 4. In section $5$ we show stability of a discretization of this linearized system and prove error estimates for the linearization. Finally, in section $6$ we verify our error estimates with a manufactured solution and illustrate our method by simulating the interdiffusion of oxygen, carbon dioxide, water vapour and nitrogen in the lungs.
 
\section{Existing numerical literature}

Despite their wide applicability, the Stefan--Maxwell equations have received relatively little attention from numerical analysts. In nearly all existing work, the equations are formulated in terms of the molar flux $N_{i}=c_{i} v_{i}$. The interdependence among the collection of driving forces implied by Gibbs--Duhem relation \eqref{FundamentalGibbsD} allows the equation for $d_n$ to be discarded. The mass-flux constraint \eqref{mass-flux constraint} is then used to eliminate the $n^\mathrm{th}$ species velocity from the system. Following this process, a non-singular matrix $\textbf{A}$ is derived which satisfies
\begin{equation}
d_{i} = \sum_{j=1}^{n-1} \textbf{A}_{ij} c_{j} v_{j} = \sum_{j=1}^{n-1} \textbf{A}_{ij} N_{j}.
\end{equation}
One can then proceed to solve for the molar fluxes in terms of the driving forces $d_{i}$ by inverting $\textbf{A}$. If, for example, we have $d_{i} = -RT\nabla c_{i}$, the inverted, truncated flux laws can be substituted into the continuity equations  \eqref{Fundamentalcontinuityeqn} for species $i = 1, ..., n-1$ to yield
\begin{equation}
\frac{\partial c_{i}}{ \partial t} = - \nabla \cdot  \Big ( \sum_{j=1}^{n-1} \textbf{A}^{-1}_{ij} \nabla c_{j} \Big ).
\end{equation}
 Thus one obtains evolution equations for the concentrations, having eliminated the molar fluxes completely. Papers which take this approach and analyse the resulting equations to determine some existence and uniqueness properties include \cite{2012Boudin, Bothe2010, Jungel2012} and \cite{Jungel2019}. \cite{2012Boudin} and \cite{Jungel2019} also analyse numerical schemes along these lines. It is worth remarking that the matrix $\textbf{A}^{-1}$ is not positive symmetric definite, although, at least in certain circumstances, one can define `entropy variables' so that the resulting system is symmetric positive definite, as carried out by \cite{Jungel2019}.

The approach of \cite{McLeod2014} does not eliminate molar fluxes, but rather solves for them in a mixed saddle point formulation. They then prove well-posedness of a linearized system consisting of three species, under some constraints on the Stefan--Maxwell diffusion coefficients. A discretization using mixed finite elements is then presented and error bounds on the linearized system are obtained. Our paper is similar in scope, but with several key differences and extensions.

First, our approach does not need any rearrangement of \eqref{mass-flux constraint} to eliminate one species, but rather incorporates the constraint via the augmented formulation \eqref{OurMSeqn}. The choice of species to eliminate is somewhat arbitrary, and with the augmentation is no longer necessary. Augmentation also exploits the symmetric positive semidefinite structure of the transport matrix and preserves permutational symmetry of the system. This will be particularly pertinent for anticipated future work where we intend to have more complex driving forces of the form
\begin{equation}
d_{i} = - c_{i} \nabla \mu_{i} + \frac{c_{i} M_{i}}{\rho} \nabla p,
\end{equation}
where $\mu_{i}$ is the electrochemical potential of species $i$ and $p$ is the pressure. These more complex driving forces render rearrangement increasingly intractable.

Second, the symmetric positive definite structure of the augmented transport matrix yields straightforward proofs of the coercivity of bilinear forms on appropriate function spaces. As a consequence, we will prove that the linearized system is well-posed in the continuous and discrete setting and derive error bounds for its discretization in the general case of $n$ species. The methodology presented in this paper also encompasses the case where individual Stefan--Maxwell diffusion coefficients may be negative.
% Furthermore our weak formulation resides in different function spaces. In particular we formulate $c_{i} 
%\in H^{1}(\Omega)$ and $v_{i} \in L^{2}(\Omega)^{d}$. This has its consequent advantages and disadvantages. In addition we also solve for the velocity $v_{i}$ instead of the molar flux $N_{i} = c_{i} v_{i}$. 

Finally, we are able to design the discrete formulation in a structure-preserving way so that the Gibbs--Duhem equation \eqref{FundamentalGibbsD} is satisfied up to machine precision, independent of mesh size. Previous works instead assume the Gibbs--Duhem equation and use it to infer the concentration of the $n^\mathrm{th}$ species in a postprocessing step.

\section{Problem formulation}

We proceed to cast the problem into variational form. Note that both sides of equation \eqref{OurMSeqn} are proportional to $RT$ and hence without loss of generality we assume that $RT=1$. Our idealized assumption on the driving forces then becomes
\begin{equation}
d_{i}  := -\nabla c_{i}, \quad i = 1,2,\dots n.
\end{equation}
In this case the Gibbs--Duhem equation \eqref{FundamentalGibbsD} reduces to
\begin{equation}
\nabla c_{\text{T}} = 0, 
\end{equation}
 i.e.~that total concentration is constant. This is also important as the constancy of $c_{\text{T}}$ is required to be consistent with the equation of state \eqref{Idealgaslaw}, which is distinct from the Gibbs--Duhem equation.  
 We assume that $u \in H^{1}(\Omega)^{d}$ and consider the boundary conditions
\begin{align} \label{Boundarydata1}
& N_{i} \cdot \textbf{n} = c_{i} v_{i} \cdot \textbf{n}= g_{i} \: \in H^{-1/2}(\Gamma_{N}) \:\:\: \text{on} \:\:\:  \Gamma_{N}, \quad i = 1,2, \dots, n,\\ \label{Boundarydata2}
& c_{i} = f_{i} >0 \: \in  H^{1/2}(\Gamma_{D})  \:\:\: \text{on} \:\:\:  \Gamma_{D}, \quad i = 1,2, \dots, n,
\end{align}
where $\textbf{n}$ is the outward facing unit normal vector and $\Gamma_{N}, \Gamma_{D}$  partition $\partial \Omega$. The equalities in \eqref{Boundarydata1}-\eqref{Boundarydata2} are to be understood in the sense of traces \citep{evans2010partial}. It is necessary to assume that $f_{i}$ is positive for each $i=1,2, \dots, n$ to avoid $\textbf{M}$ acquiring another nullspace at the boundary. Either one of $\Gamma_{N}$ and $\Gamma_{D}$ may be empty. This boundary data is assumed to satisfy
\begin{align} \label{NeumanBC consistency}
\sum_{i=1}^{n}  g_{i}M_{i}  &=  u \cdot \textbf{n} \:\:\: \text{on} \:\:\:  \Gamma_{N}, \\
\label{DirichletBC consistency}
\sum_{i=1}^{n}  f_{i}  &= C_{\text{T}} \:\:\: \text{on} \:\:\:  \Gamma_{D},
\end{align}
where $C_{\text{T}} >0$ is a constant that we will show is equal to the total concentration \eqref{Totalconcentrationdef}. These assumptions are necessary to be consistent with the Gibbs--Duhem equation \eqref{FundamentalGibbsD} and the mass-flux constraint \eqref{mass-flux constraint}. 
Under the steady-state assumption, the species continuity equations \eqref{Fundamentalcontinuityeqn} become
\begin{equation} \label{Steadystateconteq}
\nabla \cdot (c_{i} v_{i} ) = r_{i}.
\end{equation}
Therefore, we demand that the reaction rates, $r_{i} \in L^{2}(\Omega)$, satisfy
\begin{equation} \label{Reaction consistency 1}
 \sum_{i=1}^{n} r_{i} M_{i} =  \nabla \cdot u  \text{ in } \Omega 
\end{equation}
to ensure consistency of \eqref{Steadystateconteq} with \eqref{mass-flux constraint}.

We define the function space
\begin{equation}
H_{\Gamma_{D}}^{1}(\Omega) = \{ w_{i} \in H^{1}(\Omega) \: : \:  w_{i}\restrict{ \Gamma_{D}} = 0 \},
\end{equation}
and the affine function space
\begin{equation}
H_{f_{i}}^{1}(\Omega) = \{ w_{i} \in H^{1}(\Omega) \: : \: w_{i} \restrict{ \Gamma_{D}} = f_{i} \}.
\end{equation}

We can now derive the weak formulation. 
We test \eqref{OurMSeqn} with $\tau_{i} \in L^{2}(\Omega)^{d} $ and integrate over $\Omega$ to derive for all $i = 1,2, \dots, n$,
\begin{equation} \label{WeakformOurMSeqn}
\int_{\Omega} \Big ( - \nabla c_{i} +\frac{\gamma M_{i}c_{i} }{\rho} u \Big ) \cdot \tau_{i}=  \int_{\Omega} \Big ( \sum_{j \neq i}^{n}\frac{c_{i} c_{j}}{\mathcal{D}_{ij} c_{\text{T}}} \big ( v_{i} -  v_{j} \big ) + \frac{\gamma M_{i} c_{i}}{\rho} \sum_{j=1}^{n} M_{j} c_{j} v_{j} \Big ) \cdot \tau_{i},
\end{equation}
for all $\tau_{i} \in L^{2}(\Omega)^{d}$.

 For a given $w_{i} \in H^{1}_{\Gamma_{D}}(\Omega)$ we multiply both sides of \eqref{Steadystateconteq} by $-w_{i}$ and integrate by parts to yield that for all $i = 1,2, \dots, n$,
\begin{equation} \label{weakcontinuity}
\int_{\Omega} c_{i} v_{i} \cdot \nabla w_{i} - \int_{\Gamma_{N}} g_{i} w_{i}  = -\int_{\Omega} r_{i} w_{i},
\end{equation}
for all $w_{i} \in H^{1}_{\Gamma_{D}}(\Omega)$. We therefore seek $v_{i} \in L^{2}(\Omega)^{d}$ and $c_{i} \in H^{1}_{f_{i}}(\Omega)$ such that \eqref{WeakformOurMSeqn} and \eqref{weakcontinuity} hold for every $\tau_{i} \in L^{2}(Q)^{d}$ and $w_{i} \in H^{1}_{\Gamma_{D}}(\Omega)$, for each $i=1,2,\dots,n$. 

\begin{remark} \label{1Dcase}
In the case of one dimension, \eqref{Steadystateconteq} and the boundary data \eqref{Boundarydata1}-\eqref{Boundarydata2} allow us to recover $c_{i}v_{i}$ completely. Consequently no augmentation is necessary.
\end{remark}

We will now show that such a weak solution satisfies both the Gibbs--Duhem equation \eqref{FundamentalGibbsD} and the mass-flux constraint \eqref{mass-flux constraint}. 
Choosing $\tau_{i}= \tau \in L^{2}(\Omega)^{d}$ for every $i=1,2,\dots,n$ and summing over $i$ in \eqref{WeakformOurMSeqn} yields
\begin{equation}
\sum_{i=1}^{n} \int_{\Omega} \Big ( - \nabla c_{i} +\frac{\gamma M_{i}c_{i} }{\rho} u \Big ) \cdot \tau  = \sum_{i=1}^{n} \int_{\Omega} \Big ( \sum_{j \neq i}^{n}\frac{c_{i} c_{j}}{\mathcal{D}_{ij} c_{\text{T}}} \big ( v_{i} -  v_{j} \big ) + \frac{\gamma M_{i} c_{i}}{\rho} \sum_{j=1}^{n} M_{j} c_{j} v_{j} \Big ) \cdot \tau.
\end{equation}
However we can use the nullspace \eqref{Nullspace} and symmetry of $\textbf{M}$ to deduce 
\begin{equation}
\sum_{i=1}^{n}  \sum_{j \neq i}^{n} \frac{c_{i} c_{j}}{\mathcal{D}_{ij} c_{\text{T}}} \big ( v_{i} -  v_{j} \big ) = \sum_{i,j=1}^{n} \textbf{M}_{ij} v_{j} =0,
\end{equation}
and by the definition of the density \eqref{definition of density}, we obtain that
\begin{equation} \label{weakNotquiteGD}
 \sum_{i=1}^{n} \int_{\Omega} \gamma M_{i} c_{i} v_{i} \cdot \tau - \int_{\Omega} \gamma u  \cdot \tau +\int_{\Omega}  \nabla c_{\text{T}} \cdot \tau  = 0,
\end{equation}
for all $\tau \in L^{2}(\Omega)^d$. Considering the first and second terms with the choice $\tau = \nabla w$ for some $w  \in H^{1}_{\Gamma_{D}}(\Omega)$, and using \eqref{weakcontinuity},
\begin{align} \label{weakNotquiteGD2}
  \sum_{i=1}^{n}  \int_{\Omega} \gamma M_{i} c_{i} v_{i} \cdot \nabla w- \int_{\Omega} \gamma u  \cdot \nabla w &= \sum_{i=1}^{n} \gamma \Big (-\int_{\Omega}  M_{i} r_{i} w + \int_{\Gamma_{N}} M_{i} g_{i} w \Big )-\int_{\Omega} \gamma u \cdot \nabla w \\
   &= -\int_{\Omega}    \gamma w \nabla \cdot u+ \int_{\Gamma_{N}} \gamma w u \cdot \textbf{n}-\int_{\Omega} \gamma u \cdot \nabla w \quad \text{ (by \eqref{NeumanBC consistency} and \eqref{Reaction consistency 1})}\nonumber\\
   &= 0, \nonumber
\end{align}
the final equality following from  integration by parts. In light of this, \eqref{weakNotquiteGD} becomes
\begin{equation} \label{WeakenforcingofGD}
\int_{\Omega}  \nabla c_{\text{T}} \cdot \nabla w = 0,
\end{equation}
for every $w \in H^{1}_{\Gamma_{D}}(\Omega)$.
In particular, as $c_{\text{T}}$ is constant on $\Gamma_{D}$ by \eqref{DirichletBC consistency},  there exists a $w \in H^{1}_{\Gamma_{D}}(\Omega)$ such that $\nabla w = \nabla c_{\text{T}}$. For this choice of $w$, \eqref{WeakenforcingofGD} becomes
\begin{equation} \label{GDenforced}
\int_{\Omega} \big | \nabla c_{\text{T}} \big |^{2} = 0.
\end{equation}
Hence $\nabla c_{\text{T}}=0$ almost everywhere, which is the Gibbs--Duhem equation \eqref{FundamentalGibbsD}. The relationship \eqref{DirichletBC consistency} ensures that $c_{\text{T}} = C_{\text{T}}$. Equation \eqref{weakNotquiteGD} then simplifies to
\begin{equation}
\int_{\Omega} \Big ( \sum_{i=1}^{n} M_{i}   c_{i}  v_{i} \Big )\cdot \tau=\int_{\Omega} u \cdot \tau  \:\:\:\:\:\:\:\: \forall \:\: \tau \in L^{2}(\Omega)^{d},
\end{equation}
a variational statement of the mass-flux constraint \eqref{mass-flux constraint}.

\begin{remark} \label{Nonlinearnullspace}
With pure Neumann boundary data ($\Gamma_D = \emptyset$), the system \eqref{WeakformOurMSeqn}-\eqref{weakcontinuity} is not well posed. Observe that if $c_{i} $ and $v_{i}$ solve equations \eqref{WeakformOurMSeqn} and \eqref{weakcontinuity} then so do the variables $\hat{c}_{i} = \alpha c_{i}$ and $\hat{v}_{i} = \alpha^{-1} v_{i}$ for any $\alpha >0$.  In order to make the problem well posed it is necessary to impose auxiliary conditions such as
\begin{equation} \label{Making Neumann problem well posed.}
\int_{\Omega} c_{i} = \bar{C}_{i}, \quad i = 1,2, \dots, n,
\end{equation}
for known constants $\bar{C}_{i}$.
The physical interpretation of this constraint is clear. In the transient dynamics we have the continuity equations
\begin{equation}
\frac{\partial c_{i}}{\partial t} = - \nabla \cdot (c_{i} v_{i} ) + r_{i}.
\end{equation}
Integrating over $\Omega$ and using the divergence theorem we deduce that
\begin{equation}
\frac{d}{dt} \int_{\Omega} c_{i} = - \int_{\Omega} g_{i} + \int_{\Omega} r_{i}.
\end{equation}
For a steady-state solution to exist, it is necessary that the right hand side of this equation is $0$. Therefore, for all time $t$,
\begin{equation}
\frac{d}{dt} \int_{\Omega} c_{i}  = 0.
\end{equation}
Hence the integral in \eqref{Making Neumann problem well posed.} is independent of time and therefore $\bar{C}_{i}$ is completely specified by the initial conditions. 
\end{remark}

\section{Linearization and well-posedness}

We consider a linearization of Picard type.  The general approach is that whenever a velocity is multiplied by a concentration, we replace the concentration with our current guess. The exception to this is explained in Remark \ref{Remarkaboutlinearfunction}. Let us define the function spaces $X = H^{1}(\Omega)^{n}$, $X_{\Gamma_{D}} = H_{\Gamma_{D}}^{1}( \Omega)^{n}$, $Q = (L^{2}(
\Omega)^{d})^{n}$ as well as the affine function space $X_{\tilde{f}} = (H^{1}_{f_{1}}(\Omega), \dots, H^{1}_{f_{n}}(\Omega))$. We set the norm on $X_{\Gamma_{D}}$ as $\|\cdot \|_{X_{\Gamma_{D}}} = \| \cdot \|_{H_{0}^{1}(\Omega)^{n}}$.
Throughout the rest of this paper we will frequently use the notation $\tilde{q} = (q_{1},\dots,q_{n})$ to denote an $n$-tuple in one of these function/affine function spaces as well as their discrete subspaces.

Given a previous guess for the concentration $\tilde{c}^{k} = (c^{k}_{1},\dots,c^{k}_{n})$, we define a bilinear form
$a_{\tilde{c}^{k}}(\cdot, \cdot) : Q \times Q \rightarrow \mathbb{R} $ given by
\begin{equation} \label{defineAblin}
a_{\tilde{c}^{k}}(\tilde{v}, \tilde{\tau} ) =\sum_{i=1}^{n} 
\int_{\Omega} \Big ( \sum_{j \neq i}^{n} \frac{ c^{k}_{i} c^{k}_{j}}{\mathcal{D}_{ij} c_{\text{T}}} \big ( v_{i} -  v_{j} \big )  +\frac{\gamma M_{i} c^{k}_{i}}{\rho^{k}} \sum_{j=1}^{n} M_{j} c^{k}_{j}  v_{j} \Big ) \cdot \tau_{i} = \sum_{i,j}^{n} \int_{\Omega}\textbf{M}^{\gamma, k}_{ij} v_{j} \cdot \tau_{i},
\end{equation}
 for $\tilde{\tau}, \tilde{v} \in Q$. Here $\textbf{M}^{\gamma, k}$ denotes the augmented transport matrix, the $i,j$ entries being defined by using the current guess for the concentration $\tilde{c}^{k}$ in equations \eqref{transport matrix} and \eqref{augmented transport matrix}. Similarly, $\rho^{k}$ is the density evaluated using $\tilde{c}^{k}$ in \eqref{definition of density}. 
 
 For the current guess $\tilde{c}^{k}$ we also define the bilinear 
form  $b_{\tilde{c}^{k}}: Q \times X \rightarrow \mathbb{R}$,
 \begin{equation}
 b_{\tilde{c}^{k}}(\tilde{\tau}, \tilde{w}  ) = \sum_{i=1}^{n} \int_{\Omega} c^{k}_{i} \tau_{i} \cdot \nabla w_{i},
\end{equation}
for $(\tilde{\tau}, \tilde{w}) \in Q \times X$, and the bilinear form $b: Q \times X \rightarrow \mathbb{R}$, 
\begin{equation}
 b(\tilde{\tau}, \tilde{w}) = \sum_{i=1}^{n} \int_{\Omega} \tau_{i} \cdot \nabla w_{i}.
\end{equation}

 For $\tilde{\tau} \in Q$ the linear functional $l_{\tilde{c}^{k}}(\cdot ) : Q \rightarrow \mathbb{R}$ is defined as
\begin{equation} \label{linearfunctional}
l_{\tilde{c}^{k}}(\tilde{\tau}) = \gamma \sum_{i=1}^{n} \int_{\Omega} \frac{c^{k}_{i} M_{i}}{\rho^{k}} \tau_{i} \cdot u.
\end{equation}

The non-linear iteration scheme is as follows. We take an initial guess $(\tilde{v}^{0}, \tilde{c}^{0}) \in Q \times X_{\tilde{f}} $ which satisfies the Dirichlet boundary data \eqref{Boundarydata1} and
\begin{equation} \label{totalconcentrationguess}
\sum_{i=1}^{n} c^{0}_{i} = c_{\text{T}}
\end{equation}
almost everywhere for a given constant $c_{\text{T}}$,  determined by either \eqref{DirichletBC consistency} or \eqref{Making Neumann problem well posed.}. For $k=0,1,2,\dots$ the next iterate of the sequence is computed as the solution to the following generalized saddle point problem: find $(\tilde{v}^{k+1}, \tilde{c}^{k+1}) \in Q \times X_{\tilde{f}}$ such that
\begin{align} \label{Linearsaddlepointsystem1}
&a_{\tilde{c}^{k}}(\tilde{v}^{k+1}, \tilde{\tau}) + b(\tilde{\tau},  \tilde{c}^{k+1}) =l_{\tilde{c}^{k}}(\tilde{\tau}), \:\:\: \forall \: \tilde{\tau} \in Q, \\ \label{Linearsaddlepointsystem2}
& b_{\tilde{c}^{k}}( \tilde{v}^{k+1}, \tilde{w})  = -(\tilde{r}, \tilde{w})_{L^{2}(\Omega)^{n}} + (\tilde{g}, \tilde{w} )_{L^{2}(\Gamma_{N})^{n}}, \ \:\:\: \forall \: \tilde{w} \in X_{\Gamma_{D}},
\end{align}
subject to the Dirichlet conditions \eqref{Boundarydata2}. This is repeated until
\begin{equation}
\| \tilde{c}^{k+1}-\tilde{c}^{k} \|_{X} + \| \tilde{v}^{k+1}-\tilde{v}^{k} \|_{Q} \leq \varepsilon,
\end{equation} 
 for a set tolerance $\varepsilon > 0$.
 
  Note that $(\tilde{v}^{k}, \tilde{c}^{k})$ is a weak solution to the non-linear problem \eqref{WeakformOurMSeqn}-\eqref{weakcontinuity} if and only if it is a fixed point of this iteration scheme. Indeed if $(\tilde{v}^{k}, \tilde{c}^{k})$ is a weak solution to the non-linear problem \eqref{WeakformOurMSeqn}-\eqref{weakcontinuity} then the solution $(\tilde{v}^{k+1}, \tilde{c}^{k+1})$ to the equations \eqref{Linearsaddlepointsystem1}-\eqref{Linearsaddlepointsystem2} remains $(\tilde{v}^{k}, \tilde{c}^{k})$. Conversely if $(\tilde{v}^{k+1}, \tilde{c}^{k+1})=(\tilde{v}^{k}, \tilde{c}^{k})$ then, converting \eqref{Linearsaddlepointsystem1}-\eqref{Linearsaddlepointsystem2} to a non-linear system by replacing $\tilde{c}^{k}$ with $\tilde{c}^{k+1}$, we recover the non-linear problem \eqref{WeakformOurMSeqn}-\eqref{weakcontinuity} and observe it is solved with  $(\tilde{v}^{k+1}, \tilde{c}^{k+1})$.

We proceed to prove well-posedness of the linear system \eqref{Linearsaddlepointsystem1}-\eqref{Linearsaddlepointsystem2} by applying either  Theorem $2.1$ in \cite{GeneralSaddlepoint} or Theorem 3.1 in \cite{Generalisedsaddleerror1982}. 
 To invoke these theorems we shall prove the following conditions.

\textit{Condition 1}: 
There exists a constant $\alpha>0$ such that
\begin{equation} \label{Condition1}
a_{\tilde{c}^{k}}(\tilde{v},\tilde{v}) \geq \alpha \| \tilde{v} \|^{2}_{Q}
\end{equation}
for all $\tilde{v} \in Q$.

 \textit{Condition 2}:
There exist constants $\beta_{i} > 0$, $i=1,2$ such that for all $\tilde{w} \in X$,
\begin{equation}
 \begin{aligned} \label{Condition2}
\underset{\tau \in Q}{\sup} \frac{ b(\tilde{\tau}, \tilde{w})}{\| \tilde{\tau} \|_{Q}} &\geq \beta_{1} \| \tilde{w} \|_{X}, \\
\underset{\tau \in Q}{\sup} \frac{ b_{\tilde{c}^{k}}(\tilde{\tau}, \tilde{w})}{\| \tilde{\tau} \|_{Q}} &\geq \beta_{2} \| \tilde{w} \|_{X}.
 \end{aligned}
 \end{equation}
 
 \begin{remark} \label{Remarkaboutlinearfunction}
 An alternative to our definition of the linear functional \eqref{linearfunctional} would be to replace $\tilde{c}^{k}$ with $\tilde{c}^{k+1}$ and  therefore include the term as part of the bilinear functional $b(\cdot, \cdot)$ instead. However, the current formulation \eqref{Linearsaddlepointsystem1}-\eqref{Linearsaddlepointsystem2} ensures that we can derive the equivalent of \eqref{weakNotquiteGD} for the linearized system
 \begin{equation} 
 \sum_{i=1}^{n} \int_{\Omega} \gamma M_{i} c^{k}_{i} v_{i} \cdot \tau - \int_{\Omega} \gamma u  \cdot \tau +\int_{\Omega}  \sum_{i=1}^{n} \nabla c_{i}^{k+1} \cdot \tau  = 0.
\end{equation}\\
Then, following 
an argument identical to that presented in section 3, we deduce that for each $k$, the iterates satisfy
\begin{equation}
\sum_{i=1}^{n} c_{i}^{k+1} = c_{\text{T}}
\end{equation}
almost everywhere. When combined with the assumption that the concentrations are positive almost everywhere, this implies that $a_{\tilde{c}^{k}}(\cdot, \cdot), b(\cdot, \cdot), b_{\tilde{c}^{k}}(\cdot, \cdot)$ are all bounded bilinear functionals on their respective function spaces.
\end{remark}
 
 \begin{remark}
The common alternative, to formulate the problem in terms of molar fluxes rather than velocities, has the advantage that the continuity equations do not need to be linearized. However, a disadvantage is that the resulting bilinear form $a(\cdot, \cdot)$ is no longer symmetric or coercive, which would add significant difficulty to the analysis.
\end{remark}
 
 In order to prove \eqref{Condition1} it will be useful to write the bilinear form, $a_{\tilde{c}^{k}}(\cdot, \cdot)$ as the integral of a quadratic form. For this purpose it is useful to define the matrix
\begin{equation}
\mathcal{M}^{\gamma,k} = \textbf{M}^{\gamma, k} \otimes \textbf{I}
\end{equation}
where $\textbf{I}$ is the $d \times d$ identity matrix and $\otimes$ is the Kronecker product.  We can then write the bilinear form as 
\begin{equation}
a_{\tilde{c}^{k}}(\tilde{v}, \tilde{\tau} )  = \int_{\Omega} \tilde{\tau} \cdot  \mathcal{M}^{\gamma,k} \tilde{v}.
\end{equation}
To show the coercivity condition \eqref{Condition1}  we must show for some $\alpha>0$
\begin{equation}
a_{\tilde{c}^{k}}(\tilde{v}, \tilde{v} ) = \int_{\Omega} \tilde{v} \cdot  \mathcal{M}^{\gamma,k} \tilde{v} \geq \int_{\Omega} \alpha | \tilde{v} |^{2} 
\end{equation}
Hence \eqref{Condition1} is satisfied if and only if $\mathcal{M}^{\gamma,k}$ is uniformly positive definite over $\Omega$ almost everywhere.  Either by direct calculation, or by using a standard property of the Kronecker delta product, one can verify that $\mathcal{M}^{\gamma,k}$ will have the same eigenvalues as $\textbf{M}^{\gamma,k}$, each with geometric multiplicity of $d$. Therefore coercivity of the bilinear form $a_{\tilde{c}^{k}}(\cdot, \cdot)$ is equivalent to showing that $\textbf{M}^{\gamma,k}$ is symmetric positive definite almost everywhere in $\Omega$.

Assuming that every component of our current guess $\tilde{c}^{k}$ is strictly positive almost everywhere, we prove positive definiteness of $\textbf{M}^{\gamma,k}$ in the following lemma.
\begin{lemma} \label{Positivedefinitelemma}
If $c^{k}_{i} \geq \kappa >0$ a.e.~for each $i=1,2,\dots,n$ and a positive constant $\kappa$, then for any $\gamma >0$, the matrix $\textbf{M}^{\gamma,k}$ is symmetric positive definite almost everywhere.
\end{lemma}
\begin{proof}
%\pef{We should start this with `For almost every $x \in \Omega$' \dots, and make it clear that the entire argument
%is to be repeated at every point in the domain.}
For almost every $x \in \Omega$, $\textbf{M}^{k}$ is symmetric positive semidefinite. We proceed with the following argument pointwise. The normalized eigenvectors $\{\vartheta^{M}_{1},\dots,\vartheta^{M}_{n} \}$ form an orthonormal basis. By hypothesis the associated eigenvalues $\{\lambda^{\textbf{M}}_{1},\dots,\lambda^{\textbf{M}}_{n} \}$ can be ordered such that 
\begin{equation} \label{Ordering of eigenvalues}
0=\lambda^{\textbf{M}}_{1} < \lambda_{2}^{\textbf{M}} \leq \dots \leq \lambda^{\textbf{M}}_{n}.
\end{equation}
The nullspace of $\textbf{M}^{k}$ then consists of the space spanned by the vector $\vartheta^{M}_{1} = n^{-1/2}(1,1,\dots,1) \in \mathbb{R}^{n}$. Furthermore, 
 \begin{equation} \label{BoundonlowerEval}
 \lambda^{\textbf{M}}_{2} \geq \lambda_{\kappa} >0
 \end{equation} 
for a $\lambda_{\kappa}$ that depends only on $\kappa$.
 
 Given any $\tilde{\vartheta} \in \mathbb{R}^{n}$ we can expand it in terms of the basis $\{\vartheta^{M}_{1},\dots,\vartheta^{M}_{n} \}$ as
\begin{equation}
\tilde{\vartheta} = \sum_{i=1}^{n} \alpha_{i} \vartheta^{M}_{i}
\end{equation}
for basis coefficients $\{ \alpha_{i} \}_{i=1}^{n}$. Furthermore, by orthonormality,
\begin{equation} \label{Basisexpansion}
\tilde{\vartheta} \cdot \textbf{M}^{k} \tilde{\vartheta} = \sum_{i=1}^{n} \lambda^{\textbf{M}}_{i} |\alpha_{i}|^{2}.
\end{equation}
 The matrix $\mathcal{L}^{k}$ defined in \eqref{definition of L} is also symmetric positive semidefinite, explicitly for $\tilde{\vartheta} = (\vartheta_{1},\dots,\vartheta_{n}) \in \mathbb{R}^{n}$
 \begin{equation}
 \tilde{\vartheta} \cdot \mathcal{L}^{k} \tilde{\vartheta} = \frac{1}{\rho^{k}}\Big ( \sum^{n}_{j=1} M_{j}c_{j}^{k} \vartheta_{j} \Big )^{2}.
 \end{equation}
 Hence we can also construct a basis $\{\vartheta^{\mathcal{L}}_{1},\dots,\vartheta^{\mathcal{L}}_{n} \}$  of orthonormal eigenvectors.
 The vector $\vartheta^{M}_{1}$ is also an eigenvector of $\mathcal{L}^{k}$ with the eigenvalue $\rho^{k}$. We will identify this eigenvector as $\vartheta^{\mathcal{L}}_{1}$. $\mathcal{L}^{k}$ is of rank $1$ as it is the outer product of a vector with itself, and hence all other eigenvalues are zero.

 Hence for a given $\tilde{\vartheta} \in \mathbb{R}^{n}$ we can expand it as
\begin{equation}
\tilde{\vartheta} = \alpha_{1} \vartheta^{M}_{1} +\sum_{i=2}^{n} \beta_{i} \vartheta^{\mathcal{L}}_{i},
\end{equation}
for basis coefficients $\{\beta \}_{i=1}^{n}$ and calculate
\begin{equation}
\tilde{\vartheta} \cdot \mathcal{L}^{k} \tilde{\vartheta} = \rho^{k} |\alpha_{1}|^{2}.
\end{equation}

Consequently, 
\begin{equation}
\tilde{\vartheta} \cdot \textbf{M}^{\gamma, k} \tilde{\vartheta} = \gamma \rho^{k} |\alpha_{1}|^{2} + \sum_{i=2}^{n} \lambda^{\textbf{M}}_{i} |\alpha_{i}|^{2}
\end{equation}
and therefore $\textbf{M}^{\gamma, k}$ is positive definite at $x$. This argument can be repeated for every $x \in \Omega$ except perhaps on a set of measure zero. Therefore $\textbf{M}^{\gamma, k}$ is symmetric positive definite almost everywhere.
\end{proof}

\begin{remark} \label{Scalingremark}
It is useful to understand how $\lambda_{\kappa}$ scales with $\kappa$. This can be achieved by the following scaling argument. Suppose that whenever $c_{i}^{k} \geq 1$ for each $i=1,2,\dots,n$ we have the lower bound on the eigenvalues, as in \eqref{BoundonlowerEval}, of $\lambda_{\kappa = 1}$.  Now suppose that for any $\kappa >0$ we have $c^{k}_{i} \geq \kappa$ for each $i=1,2,\dots,n$. We can then define the new variables $\kappa_{i} = c^{k}_{i}/\kappa$. We then see that $\kappa_{i} \geq 1$ for each $i$. Define the $\textbf{M}_{\kappa}$ as the transport matrix with these new variables $\kappa_{i}$ replacing $c_{i}$. By direct calculation we can check that
\begin{equation} \label{Scaledmatrix}
\textbf{M}_{\kappa} = \frac{1}{\kappa} \textbf{M}.
\end{equation}

By construction we have that $\lambda^{\textbf{M}_{\kappa}}_{2} \geq \lambda_{\kappa =1}$. It follows from \eqref{Scaledmatrix} that $\lambda^{\textbf{M}}_{2} = \kappa \lambda^{\textbf{M}_{\kappa}}_{2}\geq \kappa \lambda_{\kappa =1}$. Hence we see that $\lambda_{\kappa} = {\rm O}(\kappa)$.
\end{remark}

 \begin{lemma} \label{Wellposednessconditionlemma}
 Assume that $\bar{c}_{i} \geq \kappa >0$ a.e.~for each $i = 1,2,\dots,n$ and $\gamma >0$. Then the  bilinear forms $a(\cdot, \cdot), b(\cdot, \cdot)$ and $b_{\tilde{c}^{k}}(\cdot, \cdot)$ satisfy the conditions \eqref{Condition1} and \eqref{Condition2} for some constants $\alpha, \beta_{1}, \beta_{2}$ respectively, which depend only on $\kappa$, $\Omega$.
 \end{lemma}
 
 \begin{proof}
 From Lemma \eqref{Positivedefinitelemma} we have that
 \begin{equation}
 a_{\tilde{c}^{k}}(\tilde{v}, \tilde{v} ) = \int_{\Omega} \tilde{\tau} \cdot \mathcal{M}^{\gamma, k} \tilde{v} = \sum_{i,j} \int_{\Omega} v_{j} \cdot \textbf{M}_{ij}^{\gamma,k} v_{i} \geq \alpha \|\tilde{v}\|_{Q}^{2},
 \end{equation}
where
 \begin{equation} \label{Boundonalpha}
 \alpha = \text{min} \{ \gamma  \rho^{k}, \lambda_{\kappa} \},
 \end{equation}
and $\lambda_{\kappa}$ is as in equation \eqref{BoundonlowerEval}. This proves condition \eqref{Condition1}.
 
 For conditions \eqref{Condition2}, given a $\tilde{w} \in X$, we can choose $\tilde{\tau} = \nabla \tilde{w}$ which then yields
 \begin{equation}
 b(\nabla \tilde{w}, \tilde{w} ) = \sum_{i=1}^{n} \int_{\Omega} | \nabla w_{i} |^{2} =  \| \tilde{w} \|_{X_{\Gamma_{D}}}.
 \end{equation}
 Similarly for $b_{\tilde{c}^{k}}$ we have
  \begin{equation}
 b_{\tilde{c}^{k}}(\nabla \tilde{w}, \tilde{w} ) \geq \kappa \| \tilde{w} \|_{X_{\Gamma_{D}}}.
 \end{equation}
 
 The final step is that we use either $\Gamma_{D} \neq \emptyset$ or the condition \eqref{Making Neumann problem well posed.} to deduce a Poincar\'e inequality of the form
 \begin{equation} \label{poincare}
 C_{p} \|\tilde{w} \|_{X_{\Gamma_{D}}} \geq \|\tilde{w} \|_{X} \text{ for all } \tilde{w} \in X
 \end{equation}
 for some constant $C_{p}>0$ depending only on $\Omega$. Hence
 \begin{align}
 &  b(\nabla \tilde{w}, \tilde{w} ) \geq C^{-1}_{p} \|\tilde{w} \|_{X} \\ 
& b_{\tilde{c}^{k}}(\nabla \tilde{w}, \tilde{w} ) \geq \kappa C^{-1}_{p} \|\tilde{w} \|_{X}.
 \end{align}
 \end{proof}

\begin{theorem} Assume $\gamma >0$ and the current guess $\tilde{c}^{k}$ satisfies $c^{k}_{i} \geq \kappa > 0$ a.e.~for each $i=1,2,\dots,n$ and a positive constant $\kappa$. Then, under the condition $\Gamma_{D} \neq \emptyset$ or \eqref{Making Neumann problem well posed.}, there exists a unique $(\tilde{v}^{k+1}, \tilde{c}^{k+1}) \in Q \times X_{\tilde{f}}$ which solves the system \eqref{Linearsaddlepointsystem1}-\eqref{Linearsaddlepointsystem2}.
\end{theorem}

\begin{proof} 
Our remaining obstacle for the proof is that $X_{\tilde{f}}$ is not a Hilbert space. If we use the ansatz $\tilde{c}^{k+1} = \widehat{c}^{k+1}_{0} + \tilde{c}^{0}$, where  $\widehat{c}^{k+1}_{0} \in X_{\Gamma_{D}}$ and $\tilde{c}^{0} \in X_{\tilde{f}}$ was our initial guess, then we can recast the saddle point problem \eqref{Linearsaddlepointsystem1}-\eqref{Linearsaddlepointsystem2} as: find $(\tilde{v}^{k+1} , \widehat{c}^{k+1}_{0} ) \in Q \times X_{\Gamma_{D}}$ such that
\begin{align} \label{HomoLinearsaddlepointsystem1}
&a_{\tilde{c}^{k}}(\tilde{v}^{k+1}, \tilde{\tau}) + b(\tilde{\tau},  \widehat{c}^{k+1}_{0}) =l_{\tilde{c}^{k}}(\tilde{\tau}) -  b(\tilde{\tau},  \tilde{c}^{0}) \:\:\: \forall \: \tilde{\tau} \in Q, \\ \label{HomoLinearsaddlepointsystem2}
& b_{\tilde{c}^{k}}( \tilde{v}^{k+1}, \tilde{w})  = -(\tilde{r}, \tilde{w})_{L^{2}(\Omega)^{n}} + (\tilde{g}, \tilde{w} )_{L^{2}(\Gamma_{N})^{n}}\, \:\:\: \forall \: \tilde{w} \in X_{\Gamma_{D}}.
\end{align}
By \cite[Theorem 2.1]{GeneralSaddlepoint} or \cite[Theorem 3.1]{Generalisedsaddleerror1982} there exists a unique $(\tilde{v}^{k+1}, \widehat{c}^{k+1}_{0}) \in Q \times X_{\Gamma_{D}}$ solution to this system. The proof concludes by observing that if $\tilde{c}^{k+1} = \widehat{c}^{k+1}_{0} + \tilde{c}^{0}$ then $\tilde{c}^{k+1} \in X_{\tilde{f}}$ and satisfies the system \eqref{Linearsaddlepointsystem1}-\eqref{Linearsaddlepointsystem2}.
\end{proof}

\section{Discretization and error estimates}

Here we discretize the generalized saddle point problem \eqref{Linearsaddlepointsystem1}-\eqref{Linearsaddlepointsystem2} and prove error estimates. Let $\mathcal{T}_{h}$ be a regular triangulation of $\Omega$ with maximum diameter $h$. For $m \geq 1$  we define the finite dimensional subspaces,
\begin{align} \label{discretespaceQ}
& Q_{h} = \{ \tilde{\tau}_{h} \in Q \: | \: \left.\tau_{h,i}\right|_K \in P^{m-1}(K) \:\: \forall \: K \in \mathcal{T}_{h}, \:\: i=1,2,\dots,n \}, \\ \label{discretespaceX}
& X_{h} = \{ \tilde{w}_{h} \in X \: | \: \left.w_{h,i}\right|_K \in P^{m}(K) \:\: \forall \: K \in \mathcal{T}_{h}, \:\: i=1,2,\dots,n \},
\\ \label{discretespaceX0}
& X_{\Gamma_{D},h} = \{ \tilde{w}_{h} \in X_{\Gamma_{D}} \: | \: \left.w_{h,i}\right|_K \in P^{m}(K) \:\: \forall \: K \in \mathcal{T}_{h}, \:\: i=1,2,\dots,n \}.
\end{align}
Here $P^{m}(K)$ denotes the set of $m^{\mathrm{th}}$ order polynomials on the cell $K \in \mathcal{T}_{h}$.

We will require linear interpolation operators on the spaces $X$ and $Q$, see \cite[pp~72]{boffi2013mixed}.
\begin{proposition} \label{errorestimateproposition}
There exist linear interpolation operators $\Pi_{h} : X \rightarrow X_{h}$ and $\Lambda_{h}: Q \rightarrow Q_{h}$ and constants $C_{1}, C_{2}$ such that, for any $\tilde{c} \in X$, $\tilde{v} \in Q$,
\begin{align*}
& \| \tilde{c} - \Pi_{h} \tilde{c} \|_{X} \leq C_{1} h^{m} \| \tilde{c} \|_{H^{m+1}_{0}(\Omega)^{n}}, \\
& \| \tilde{v} - \Lambda_{h} \tilde{v} \|_{Q} \leq C_{2} h^{m}  \| \tilde{v} \|_{(H^{m}_{0}(\Omega)^{d})^{n}}.  
\end{align*}
\end{proposition}

 Our non-linear iteration scheme in the discrete case is as follows; we take an initial guess $\tilde{c}^{0} \in X_{\tilde{f}}$ which satisfies \eqref{totalconcentrationguess} and then construct 
 $\tilde{c}^{0}_{h} := \Pi_{h} \tilde{c}^{0} \in X_{h}$. The Dirichlet boundary conditions \eqref{Boundarydata2} are typically only satisfied approximately; however we note that, due to linearity of the interpolation operator and equation \eqref{totalconcentrationguess},
 \begin{equation} \label{Interpolationpreserves}
 \sum_{i=1}^{n} c^{0}_{i,h} =  
 \sum_{i=1}^{n} \Pi_{h} c^{0}_{i} = \Pi_{h} c_{\text{T}} = c_{\text{T}},
 \end{equation}
and therefore condition \eqref{DirichletBC consistency} remains enforced.

 For $k=0,1,2,\dots$ the next iterate of the sequence $(\tilde{v}^{k+1}_{h}, \tilde{c}^{k+1}_{h})$ is computed by solving the following linear system: find $(\tilde{v}_{h}^{k+1}, \widehat{c}^{k+1}_{0,h}) \in  Q_{h} \times X_{\Gamma_{D},h} $ such that % \eqref{HomoLinearsaddlepointsystem1}-\eqref{HomoLinearsaddlepointsystem2}, 
\begin{align} \label{DiscreteLinearsaddlepointsystem1}
&a_{\tilde{c}^{k}_{h}}(\tilde{v}_{h}^{k+1}, \tilde{\tau}_{h}) + b(\tilde{\tau}_{h},  \hat{c}_{0,h}^{k+1}) =l_{\tilde{c}_{h}^{k}}(\tilde{\tau}_{h}) -  b(\tilde{\tau}_{h},  \tilde{c}^{0}_{h}) \:\:\: \forall \: \tilde{\tau}_{h} \in Q_{h}, \\ \label{DiscreteLinearsaddlepointsystem2}
& b_{\tilde{c}^{k}_{h}}( \tilde{v}^{k+1}_{h}, \tilde{w}_{h})  = -(\tilde{r}_{h}, \tilde{w}_{h})_{L^{2}(\Omega)^{n}} + (\tilde{g}, \tilde{w}_{h})_{L^{2}(\Gamma_{N})^{n}} \:\:\: \forall \: \tilde{w}_{h} \in X_{\Gamma_{D},h}.
\end{align}
We then set $\tilde{c}^{k+1}_{h} = \widehat{c}^{k+1}_{0,h} + \tilde{c}^{0}_{h} $ and repeat this until $\|\tilde{c}^{k+1}_{h} - \tilde{c}^{k}_{h} \|_{X} + \|\tilde{v}^{k+1}_{h} - \tilde{v}^{k}_{h} \|_{Q} \leq \varepsilon$ for our tolerance $\varepsilon > 0$.

A distinct advantage of our formulation is that the coercivity condition \eqref{Condition1} and the inf-sup condition \eqref{Condition2} are automatically satisfied with the same constants $\alpha, \beta_{1}, \beta_{2}$. This follows from the fact that the choice of function spaces preserves a crucial structure:
\begin{equation} \label{neededforinfsup}
\text{for any } \tilde{w} \in X_h, \:\nabla \tilde{w} \in Q_{h},
\end{equation} 
which in particular allows us to repeat the proofs of \eqref{Condition1} and \eqref{Condition2} in the discrete setting in exactly the same manner. We thus have the following.
\begin{theorem}
Assume $\gamma >0$ and that $\tilde{c}^{k}_{h}$ satisfies $ c^{k}_{i,h} \geq \kappa > 0$ a.e.~for each $i=1,2,\dots,n$ and a positive constant $\kappa$. Then, under the condition $\Gamma_{D} \neq \emptyset$ or \eqref{Making Neumann problem well posed.}, there exists a unique $(\tilde{v}^{k+1}_{h}, \hat{c}^{k+1}_{0,h}) \in Q_{h} \times X_{\Gamma_{D}, h}$ which solves the system \eqref{DiscreteLinearsaddlepointsystem1}-\eqref{DiscreteLinearsaddlepointsystem2}.
\end{theorem}

%\begin{remark}
%Although the inf-sup condition comes readily in our discretisation, it is important that the condition \eqref{neededforinfsup} holds. In unreported numerical experiments we investigated many discretisations; we found no stable discretisation that violates \eqref{neededforinfsup}. 
%\end{remark}

Given the well-posedness of the discretized system, we  proceed to obtain error estimates. However, given that we have the conditions \eqref{Condition1}-\eqref{Condition2} satisfied for the spaces $Q_{h}$ and $X_{h}$, we can use a known result for generalized saddle point systems \cite[Theorem 4.1]{Generalisedsaddleerror1982} to deduce the following.
\begin{theorem} \label{Galerkinorthogonality estimate}
 There exist constants $L_{1}, L_{2}$ depending only on $\alpha, \beta_{1},\beta_{2}, \Omega$ such that
\begin{align}
& \| \widehat{c}^{k+1}_{0} - \widehat{c}^{k+1}_{0,h} \|_{X_{\Gamma_{D}}} \leq L_{1} \Big ( \underset{\tilde{w}_{h} \in X_{\Gamma_{D}, h}}{\inf} \| \widehat{c}^{k+1}_{0} - \tilde{w}_{h} \|_{X} + \underset{\tilde{\tau}_{h} \in Q_{h}}{\inf}  \| \tilde{v}^{k+1} - \tilde{\tau}_{h} \|_{Q} \Big ), \\
& \| \tilde{v}^{k+1} -\tilde{v}^{k+1}_{h} \|_{Q} \leq L_{2}  \Big ( \underset{\tilde{w}_{h} \in X_{\Gamma_{D},h}}{\inf} \| \widehat{c}^{k+1}_{0} - \tilde{w}_{h} \|_{X} + \underset{\tilde{\tau}_{h} \in Q_{h}}{\inf}  \| \tilde{v}^{k+1} - \tilde{\tau}_{h} \|_{Q} \Big ).
\end{align}
\end{theorem}

We have by the Poincar\'e inequality \eqref{poincare} and Proposition \ref{errorestimateproposition}, for some constants $C_{p}, C_{1}$,
\begin{align}
& \|\tilde{c}^{k+1} - \tilde{c}^{k+1}_{h} \|_{X}  \leq  \|\widehat{c}^{k+1}_{0} - \widehat{c}^{k+1}_{0,h} \|_{X}+ \|\tilde{c}^{0} - \tilde{c}^{0}_{h} \|_{X}, \\
& \leq C^{-1}_{p} \|\widehat{c}^{k+1}_{0} - \widehat{c}^{k+1}_{0,h} \|_{X_{\Gamma_{D}}}+ C_{1} h^{m} \| \tilde{c}^{0} \|_{H^{m+1}_{0}(\Omega)}.
\end{align}
Therefore, noting that $\Pi_{h} \widehat{c}^{k+1}_{0} \in X_{\Gamma_{D},h}$, we can combine
 Theorem \ref{Galerkinorthogonality estimate} and Proposition \ref{errorestimateproposition} to deduce the following corollary.

\begin{corollary}
 There exist constants $\bar{C}_{1}, \bar{C}_{2}$ depending only on $\alpha, \beta_{1},\beta_{2}, \Omega$ such that
\begin{align}
& \| \tilde{c}^{k+1} - \tilde{c}^{k+1}_{h} \|_{X} \leq \bar{C}_{1} h^{m}  \Big ( \| \tilde{c}^{0} \|_{H^{m+1}_{0}(\Omega)^{n}}+\| \hat{c}^{k+1}_{0} \|_{H^{m+1}_{0}(\Omega)^{n}} +\| \tilde{v}^{k+1} \|_{(H^{m}_{0}(\Omega)^{d})^{n}} \Big ), \\
& \| \tilde{v}^{k+1} -\tilde{v}_{h}^{k+1} \|_{Q} \leq \bar{C}_{2}  h^{m} \Big ( \| \tilde{c}^{0} \|_{H^{m+1}_{0}(\Omega)^{n}} +\| \hat{c}^{k+1}_{0} \|_{H^{m+1}_{0}(\Omega)^{n}}+  \| \tilde{v}^{k+1} \|_{(H^{m}_{0}(\Omega)^{d})^{n}} \Big ).
\end{align}
\end{corollary}
%\pef{Can we be clear about the dependencies of the constants? We should explicitly state that they are independent of the mesh size. Ideally we should also try to figure out how they scale with $\kappa$.}

For example, if we choose $m=1$ then we have
\begin{equation}
\|\tilde{v}^{k+1} - \tilde{v}_{h}^{k+1} \|_{Q} + \| \tilde{c}^{k+1} - \tilde{c}_{h}^{k+1} \|_{X} = {\rm O} (h).
\end{equation}
In the next section we will observe that actually  $\| \tilde{c}^{k+1} - \tilde{c}_{h}^{k+1} \|_{L^{2}(\Omega)^{n}}  = {\rm O} (h^{m+1})$. Thus it is likely that one can use duality methods to improve the error estimate in the $L^{2}$ norm of $\tilde{c}$.

\begin{remark}
It can be observed in the proof of \cite[Theorem 4.1]{Generalisedsaddleerror1982} that the constants $L_{1}, L_{2}$ appearing in Theorem \ref{Galerkinorthogonality estimate} scale as ${\rm O}(\alpha^{-1})$. Therefore from \eqref{Boundonalpha} and the scaling argument in Remark \ref{Scalingremark} we then see that the constants $\bar{C}_{1}, \bar{C}_{2}$ will scale as ${\rm O}(\kappa^{-1})$.
\end{remark}

The Gibbs--Duhem equation is preserved up to machine-precision as can observed by the following argument. Replacing $c_{\text{T}}$ with $c_{\text{T},h}$ we can reproduce the argument of section $3$ and derive the equivalent of equation 
\eqref{GDenforced};
\begin{equation}
\int_{\Omega} |\nabla c_{\text{T},h}|^{2}=0.
\end{equation}
Combining this with \eqref{Interpolationpreserves} we see that $c_{\text{T},h}=C_{\text{T}}$, where $C_{\text{T}}$ is determined by either \eqref{DirichletBC consistency} or \eqref{Making Neumann problem well posed.}. This calculation does not use any approximation based on the mesh size.

\section{Numerical results}

 Two numerical simulations were implemented with our method.  The discretization was implemented using the Firedrake software \citep{Rathgeber2016} and PETSc \citep{petsc-user-ref, petsc-efficient, Dalcin2011,Chaco95}. The arising linear systems were solved using MUMPS \citep{MUMPS01, MUMPS02}.

\subsection{Numerical example one: Manufactured solution} 

We first consider a test case on $\Omega = [0,1]^{2}$ for which the solution is analytically known in order to validate the error estimates of section $5$. 

 For $n=4$ the family of manufactured solutions is constructed as follows. For $j=1,2$ let $k_{j}(\cdot): 
\Omega \rightarrow \mathbb{R}$ be a differentiable function with a strict bound $|k_{j}|< K_{j}$ for a positive constant $K_{j}$.  We set
\begin{align*}
& c_{1} = k_{1}+K_{1}, \:\:\:\: c_{2} = -k_{1} + K_{1},\\
& c_{3} = k_{2}+K_{2} , \:\:\:\:  c_{4} = -k_{2} + K_{2}.
\end{align*}

 We further assume that
\begin{equation}
\mathcal{D}_{13}=\mathcal{D}_{14}=\mathcal{D}_{24}=\mathcal{D}_{23}.
\end{equation}
Then for any given mass-flux $u \in L^{2}(\Omega)^{d}$ an exact solution is given when
\begin{align*}
& v_{1} = -\frac{2}{RT} \Big ( \frac{K_{1}}{\mathcal{D}_{12}} + \frac{K_{1}}{\mathcal{D}_{13}} 
\Big ) \nabla \ln c_{1} + \frac{u}{c_{\text{T}}}, \:\:\:\:  v_{2} = -\frac{c_{1}}{c_{2}} v_{1} + \frac{u}{c_{\text{T}}},  \\
& v_{3} = - -\frac{2}{RT} \Big ( \frac{K_{2}}{\mathcal{D}_{34}} + \frac{K_{1}}{\mathcal{D}_{31}} 
\Big ) \nabla \ln c_{3}+ \frac{u}{c_{\text{T}}}, \:\:\:\: v_{4} = -\frac{c_{3}}{c_{4}} v_{3} + \frac{u}{c_{\text{T}}},
\end{align*}
 and, for $i=1,2,3,4$,
\begin{equation}
r_{i} =  
\text{div} \big ( c_{i} v_{i} \big ).
\end{equation}
 We then choose $M_{i} = 1$ for $i=1,2,3,4$ so that the mass-flux constraint \eqref{mass-flux constraint} is satisfied. 
 
For this numerical experiment we take $RT=1 $ and 
\begin{equation}
k_{1}(x,y) = \frac{1}{2} \exp(8xy(1-y)(1-x)), \:\:\:\: k_{2}(x,y) = \frac{1}{2} \sin(\pi x ) \sin (\pi y);
\end{equation}
 we can then take $K_{1}=K_{2}=1$. We then have $c_{\text{T}}=4$. For $i=1,2,3,4$ we pose the Dirichlet boundary conditions
\begin{equation}
c_{i} = 1, \:\:\: \text{on} \:\: \partial \Omega,
\end{equation}
and set the mass-flux $u = (0, 1)^\top$.

The diffusion coefficients are chosen as $\mathcal{D}_{12}= \mathcal{D}_{21}= 2$, $\mathcal{D}_{34} = \mathcal{D}_{43} = 3$ and all other diffusion coefficients set to $1$. We take $m=1$ for the discrete spaces \eqref{discretespaceQ}-\eqref{discretespaceX}.
For our initial guess we choose $c^{0}_{i} = c_{i,h}^{0}=1$ for $i=1,2,3,4$. We then proceed with the iteration detailed in section 5 and compute the sequence $(\tilde{v}^{k+1}_{h}, \tilde{c}_{h}^{k+1})$ until,

\begin{equation} \label{Tolerancecriteria}
 \|\tilde{c}^{k+1}_{h} - \tilde{c}^{k}_{h} \|_{X} + \|\tilde{v}^{k+1}_{h} - \tilde{v}^{k}_{h} \|_{Q} \leq \varepsilon,
\end{equation}
and for this $k$ we set $(\tilde{v}, \tilde{c} ) = (\tilde{v}^{k+1}, \tilde{c}^{k+1})$. In this experiment we took $\varepsilon= 10^{-13}$ and $\gamma =1$. The resulting concentration profile and velocity vector field for species 1 are plotted in Figure \ref{2D plot}. 

\begin{figure} 
\begin{subfigure}[t]{.55\textwidth}  
\centering
\includegraphics[width=1\linewidth]{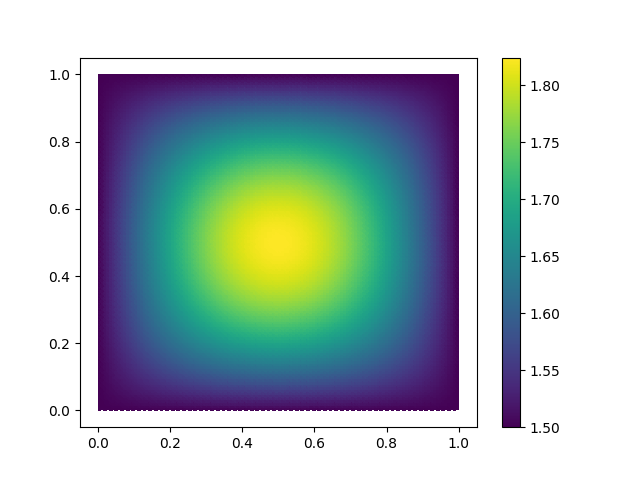}
\end{subfigure}
\hfill
\hfill
\hfill
\hfill
\begin{subfigure}[t]{.55\textwidth} 
\centering
\includegraphics[width=1\linewidth]{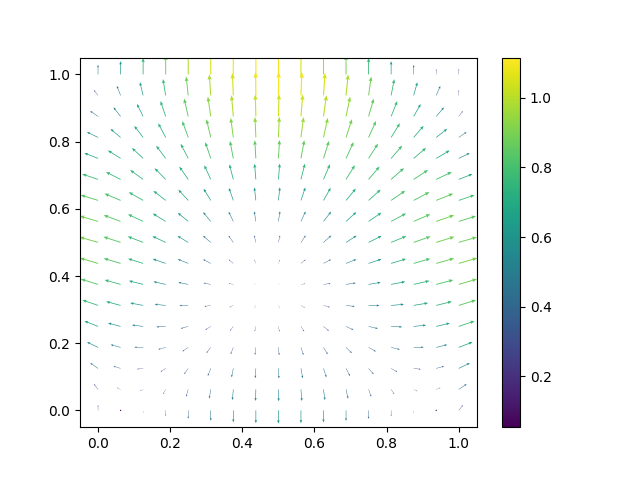}
\end{subfigure}
\caption{Concentration of species 1 (left) and its velocity vector field (right). The colour bar on the vector field plot denotes magnitude.}
 \label{2D plot} \end{figure}
 To analyse the rate of convergence we define the three errors
 \begin{align}
 & E_{1} =\Big ( \sum_{j=1}^{n} \|c_{j} - c_{j,h} \|_{L^{2}(\Omega)}^{2}\Big)^{\frac{1}{2}}, \\ 
 & E_{2} = \Big ( \sum_{j=1}^{n} \|\nabla c_{j} - \nabla c_{j,h} \|^{2}_{L^{2}(\Omega)^{d}}\Big)^{\frac{1}{2}}, \\ 
 & E_{3} = \Big( \sum_{j=1}^{n} \|v_{j} - v_{j,h} \|^{2}_{L^{2}(\Omega)^{d}} \Big )^{\frac{1}{2}}, 
\end{align}
and the error in the mass-flux
\begin{equation}
  E_{4} = \Big \| \sum_{j=1}^{n} M_{j} c_{j} v_{j} - u \Big \|_{L^{2}(\Omega)^{d}}.
 \end{equation}

According to Proposition \ref{errorestimateproposition}, $E_{j} = {\rm O}(h)$ for $i=1,2,3$. This is validated on the log-log error plot displayed in Figure  \ref{Log-Log-error}. We also observe that $E_{4} = {\rm O}(h)$.
% 
%  \medskip
%\begin{subfigure}[t]{.4\textwidth}
%\centering
%\includegraphics[width=1\linewidth]{Log-Log-error.png}
%\caption{o}
%\end{subfigure}
%\hfill 
%\begin{subfigure}[t]{.4\textwidth}
%\centering
%
%\includegraphics[width=1\linewidth]{Log-Log-error.png}
%\caption{o}
%\end{subfigure}
%\end{figure}

%
%
%\textit{Numerical experiment 2: Physical example}
%
%\bigskip
%
%In this experiment we conduct a more physically motivated simulation. We consider a cross junction geometry as depicted in figure \ref{Crosspicture}. In this geometry we have four inlets in which we prescribe an influx

\begin{center}
\centering 
\includegraphics[width=0.7 \hsize ]{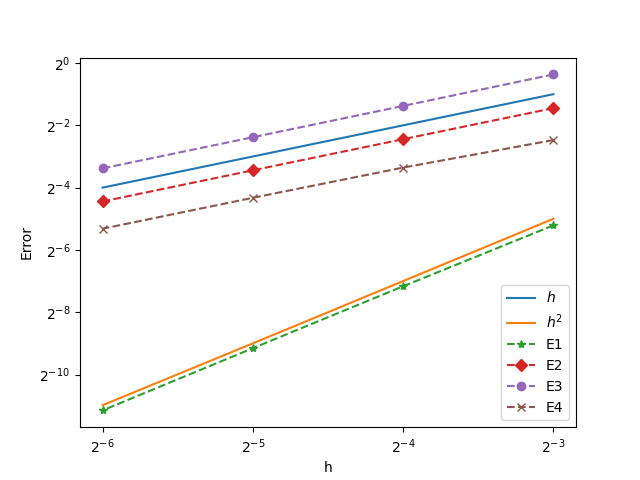}
\captionof{figure}{log-log error plot} 
\label{Log-Log-error}
\end{center}
Note that we actually observe that $E_{1} = {\rm O}(h^{2})$. This suggests that if one developed a duality-based error estimate for generalized saddle point systems, the error estimate on $E_{1}$ could be improved. 

Our discretization also preserves the Gibbs--Duhem relationship up to machine precision, independent of the mesh size.  The relevant values are tabulated in Table \ref{tab:gibbs--duhem}. 

\begin{table}[h!]
  \begin{center}
    \begin{tabular}{c |c |c} % <-- Changed to S here.
      %\toprule
      \text{Mesh size} & \text{Non-linear iterations} & $\| \nabla c_{\text{T}} \|_{L^{2}(\Omega)^{d}}$   \\
      \hline
      $8 \times 8$  & 11 & $<10^{-14}$  \\
      $16 \times 16 $ & 11 &$<10^{-14}$  \\
      $32 \times 32$  & 11 & $
<10^{-14}$ \\
      $64 \times 64$ & 11 & $<10^{-14}$  
      %\bottomrule
    \end{tabular}
  \end{center}
  \caption{The Gibbs--Duhem relationship is preserved regardless of mesh size}
  \label{tab:gibbs--duhem}
\end{table}

%%This was the only way I could get it to the bottom. If there is a better to do this let me know.
\raggedbottom

\subsection{Numerical example two: Diffusion of oxygen and effusion of carbon dioxide in the lungs}
If treated as a steady diffusion process, mass transport in the bronchi within the lungs involves simultaneous ingress of oxygen and egress of carbon dioxide. Moreover, the air through which these species diffuse also contains nitrogen and water vapour. For most modelling purposes, it is not necessary to distinguish among the various constituents of air, but in lung modelling we are interested in the distributions of both the oxygen consumed and carbon dioxide produced by the body, as well as the relative humidity along their diffusion paths. The concentrations of these compounds throughout the lungs has been modelled using the Stefan--Maxwell equations in \cite{Boudin2010} and \cite{CHANG1975109}.
For this example we solve for the mole fraction $y_{i} = c_{i} /c_{\text{T}}$. Mathematically this is the same as normalising the total concentration to $1$. As $c_{\text{T}}$ is a constant in this setting, this does not change the weak formulation or the algorithm. We take the mass-flux, $u$, as zero, and thus consider purely diffusional forces.
For a realistic lung model it would be necessary to model the transient dynamics as well as the convective forces and pressure-driven elastic expansion, but this example suffices to illustrate the time-averaged multispecies transport physics.

This simulation was computed on the mesh shown in Figure \ref{Lungmeshfigure}. The surface mesh was provided by C.~Geuzaine and J.~F.~Remacle \citep{Lungmeshsurface1, Lungmeshsurface2}, and from this the 3D mesh was constructed using the software MeshMixer \citep{Meshmixer} and Gmsh \citep{GMSH}. The mesh consisted of $115609$ vertices and $404174$ elements.

\begin{center}
\centering 
\includegraphics[width=0.7 \hsize ]{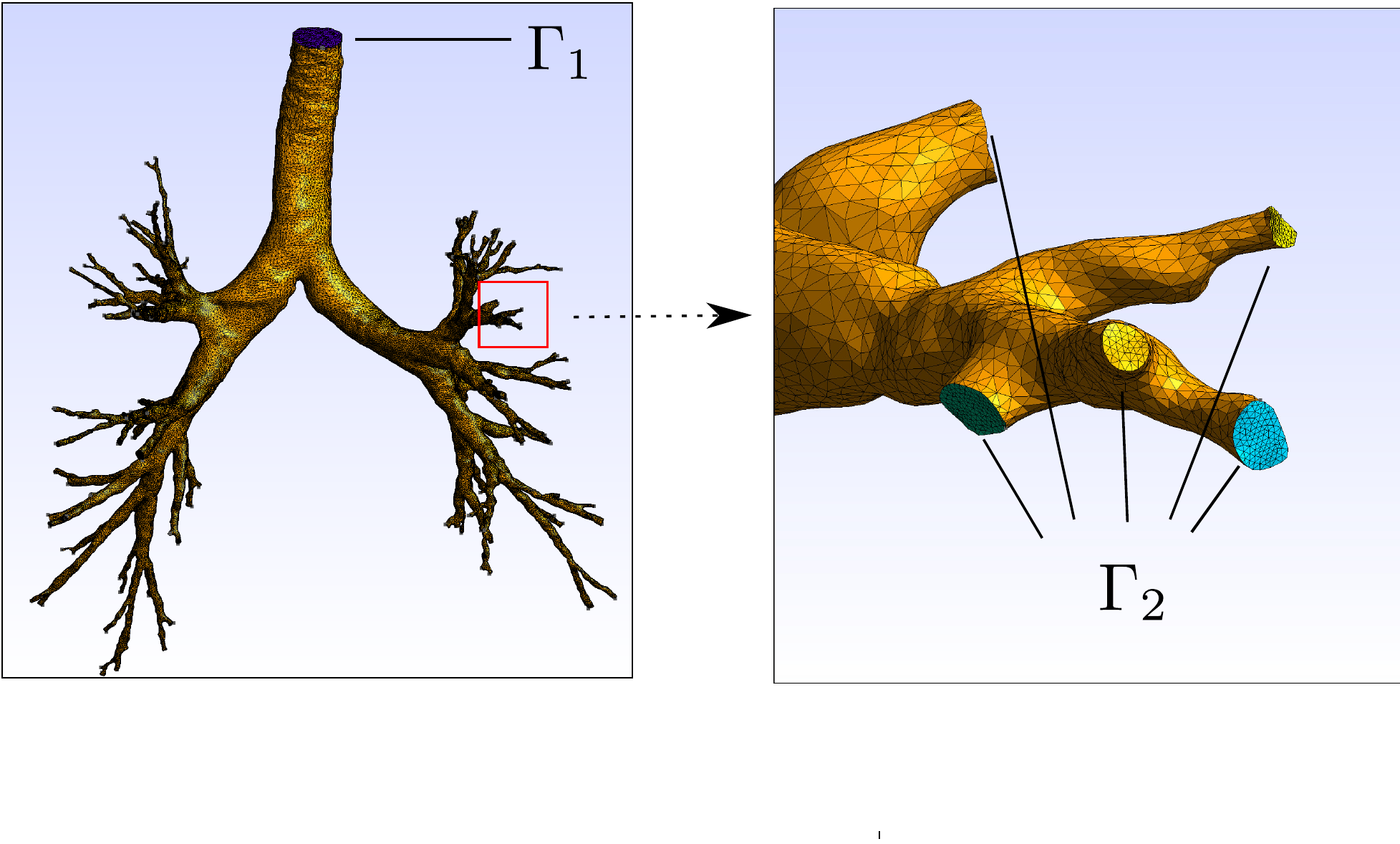}
\captionof{figure}{Mesh of the void space within the lungs at ambient pressure. The surface $\Gamma_{1}$ denotes the inlet at the trachea; the surface $\Gamma_{2}$ is a grouping of all the surfaces at the end of the tertiary bronchi.} 
\label{Lungmeshfigure}
\end{center}

Following the two-dimensional numerical experiments performed in \cite{Boudin2010}, we take mixed Neumann-Dirichlet boundary conditions. At the inlet of the trachea, $\Gamma_{1}$, and at the end of the tertiary bronchi, $\Gamma_{2}$, we set the Dirichlet boundary data to the compositions of humidified air and alveolar air respectively. For the remaining boundary region we set homogeneous Neumann (no-flux) conditions. The Stefan--Maxwell coefficients and the boundary data for this experiment, both taken from \cite{Boudin2010}, are tabulated in Tables \ref{Stefan--Maxwell coefficients} and \ref{BCtable}.

\begin{table}[h!]
    \caption{Values of the Stefan--Maxwell diffusion coefficients $\mathcal{D}_{ij}$ between species (\si{mm^{2}.s^{-1})}}
  \begin{center}
    \begin{tabular}{c |c c c  c} % <-- Changed to S here.
      %\toprule
      \text{Species} & \ce{N2} & \ce{O2} & \ce{CO2} & \ce{H2O} \\
      \hline
       
      \ce{N2} &  & 21.87 & 16.63 & 23.15 \\
      \ce{O2} & 21.87 &  & 16.40 & 22.85  \\
      \ce{CO2} & 16.63 &16.40 & &  16.02 \\
      \ce{H2O}  & 23.15 & 21.87 & 16.02  \\
%      \bottomrule
      \end{tabular}
      \end{center}
  		\label{Stefan--Maxwell coefficients}
      \end{table}

\begin{table}[h!]
    \caption{Dirichlet boundary data at the entrance of the trachea ($\Gamma_{1}$) and the end of the tertiary bronchi ($\Gamma_{2}$). Note that the air is humidified such that the water vapour mole fraction is equal at both $\Gamma_{1}$, $\Gamma_{2}$ }
  \begin{center}
    \begin{tabular}{c |c c c  c} % <-- Changed to S here.
      %\toprule
       & \ce{N2} & \ce{O2} & \ce{CO2} & \ce{H2O} \\
      \hline
       
      Mole fraction at $\Gamma_{1}$ & 0.7409 & 0.1967 & 0.0004 & 0.0620 \\
      Mole fraction at $\Gamma_{2}$ & 0.7490 & 0.1360 & 0.0530 & 0.0620  \\
%      \bottomrule
      \end{tabular}
      \end{center}
  	  \label{BCtable}
      \end{table}
      
  As there are no reactions among the species in the lung, we have $r_{i} = 0$ for each $i=1,2,3,4$. The solving parameters were set as $\varepsilon = 10^{-11}$ and $\gamma = 1$. Following our algorithm from section $5$, convergence was achieved in $12$ non-linear iterations. Each linear system had $5,312,524$ degrees of freedom and was solved on 12 cores. We remark that despite the very low concentration of carbon dioxide at $\Gamma_{1}$,  convergence was achieved in few iterations, and the mole fraction remained positive across all iterations.
  
\begin{center}
\centering  
\includegraphics[width=0.85 \hsize ]{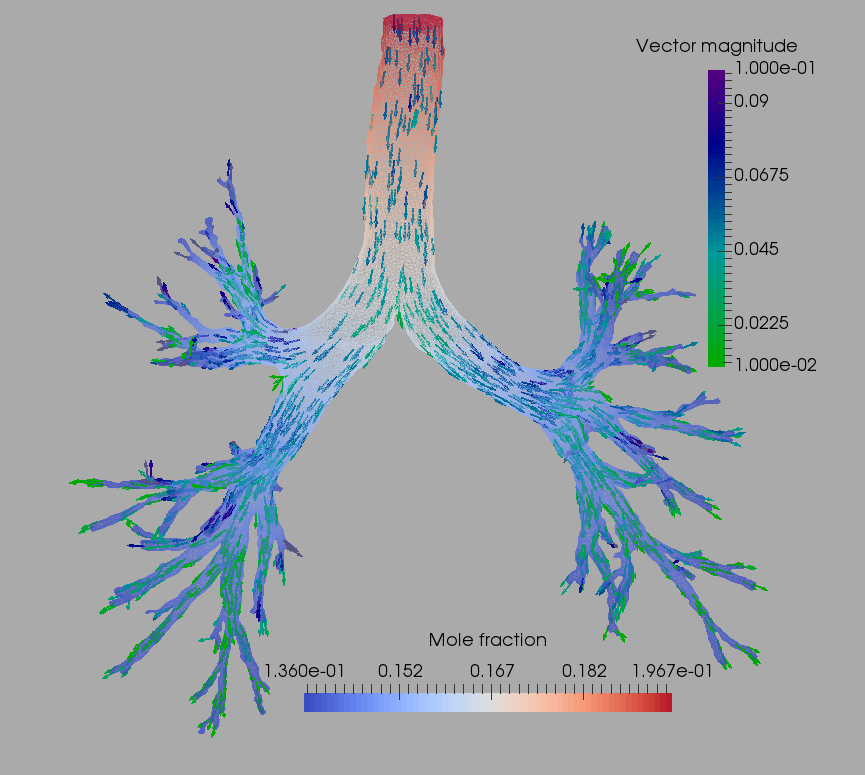}
\captionof{figure}{A  plot of the distribution of oxygen in the lungs with its velocity vector field (\si{mm.s^{-1}}).} 
\label{Oxygenprofile}
\end{center}

 Interesting physical effects are revealed by the diffusional drag forces in the water vapour. Since the mole fractions for water vapour on the boundaries $\Gamma_{1}$ and $\Gamma_{2}$ are the same, any concentration gradient of water vapour is a consequence of diffusional interactions with the other species.

Figure \ref{Waterprofile} shows modest uphill diffusion of water vapour at the trachea, where the velocity points in the same direction as the mole-fraction gradient. This can be explained as follows. The difference in the mole fractions of oxygen and carbon dioxide between the trachea and the tertiary bronchi creates a strong mole-fraction gradient, which in turn drives the velocity fields of the respective species in opposing directions. These velocity fields interact with the water vapour and attempt to drag the water vapour along with them, but the diffusional drag force exerted by \ce{CO2} on \ce{H2O} exceeds the drag by \ce{O2} on \ce{H2O}. Consequently, the water vapour tends to be dragged along with the carbon dioxide --- the \ce{H2O} velocity flows up the trachea.

\begin{center}
\centering  
\includegraphics[width=0.85 \hsize ]{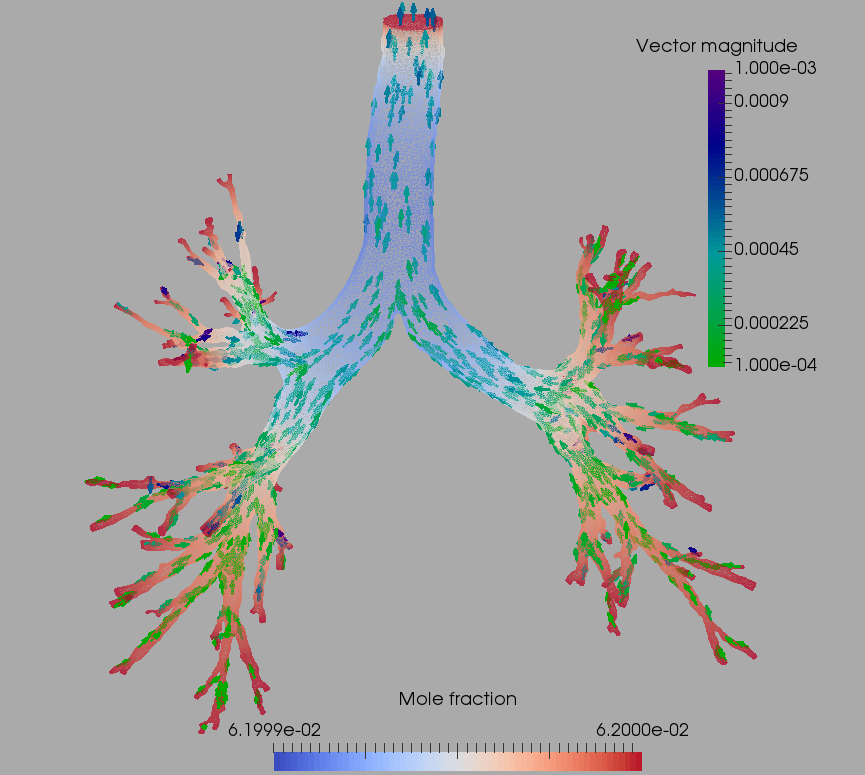}
\captionof{figure}{A plot of the distribution of water vapour in the lungs with its velocity vector field (\si{mm.s^{-1}}).} 
\label{Waterprofile}
\end{center}

 \subsection{Code availability}
 For  reproducibility,  the exact software versions used to produce the results in this paper, along with instructions for installation, has been archived at \url{https://zenodo.org/record/3860438}. The exact scripts used to produce each numerical experiment can be found at \url{https://bitbucket.org/AlexanderVanBrunt/maxwell-stefan-diffusion-equations-repository} along with the mesh used for the lungs.

\section{Conclusion}

 We derived a structure-preserving discretization of the steady-state Stefan--Maxwell diffusion problem based on an augmented saddle point formulation. The inf-sup conditions for the linearized continuous and discrete systems fundamentally rely on the symmetric positive definite structure of an augmented transport matrix, which follows from thermodynamical principles and the construction of the augmentation involving the mass-flux. 
Error estimates for the general case of $n$ species were then deduced, which were confirmed with numerical experiments.

This work considers idealized assumptions; many real-world applications require the relaxation of these assumptions. Future work will likely involve incorporate solving for momentum and including more complex driving forces. We hope that the results presented in this paper for the idealized case can provide guiding principles for a more general setting.

\section{Acknowledgements}

This work was supported by the Engineering and Physical Sciences Research Council Centre for Doctoral Training in Partial Differential Equations: Analysis and Applications (EP/L015811/1), Engineering and Physical Sciences Research Council (EP/R029423/1); the Clarendon fund scholarship; and the Faraday institution SOLBAT project and Multiscale Modelling projects, (subawards FIRG007 and FIRG003 under grant EP/P003532/1).  The authors would also like to thank C.~Geuzaine and J.~F.~Remacle for providing the surface of the mesh used in the second numerical example. 
\bibliographystyle{IMANUM-BIB}

\bibliography{Augmented-saddle-point-Maxwell-StefanArXiv}

\end{document}